\pdfoutput=1
\date{} % %

\title{The near-critical planar FK-Ising model}
\author{Hugo Duminil-Copin \and Christophe Garban \and G\'abor Pete}

\documentclass[a4paper,11pt]{article}

\usepackage[latin1]{inputenc}
\usepackage[T1]{fontenc}

\usepackage{amsmath}
\usepackage{amssymb}
\usepackage{amsthm}
\usepackage{amsfonts}
\usepackage{graphicx}
\usepackage{wasysym}

\usepackage{wrapfig}

\usepackage{color}
\usepackage{bbm,verbatim}

\usepackage{multirow} %pour les tableaux
\usepackage{array}
\newcolumntype{M}[1]{>{\centering}m{#1}}

\usepackage[a4paper,vmargin={3cm, 3.5cm},hmargin=3cm]{geometry}

\newif\ifhyper\IfFileExists{hyperref.sty}{\hypertrue}{\hyperfalse}  %%HYPERREF
\ifhyper\usepackage{hyperref}

\input epsf.sty

\newif\iffigures\figurestrue
\figuresfalse

\numberwithin{equation}{section}
\numberwithin{figure}{section}
\newtheorem{theorem}{Theorem}
\numberwithin{theorem}{section}
\newtheorem{corollary}[theorem]{Corollary}
\newtheorem{lemma}[theorem]{Lemma}

\newtheorem{proposition}[theorem]{Proposition}
\newtheorem{question}{Question}
\newtheorem{conjecture}[theorem]{Conjecture}

\newtheorem{definition}[theorem]{Definition}
\theoremstyle{remark}\newtheorem{remark}[theorem]{Remark}

\theoremstyle{plain}

\newtheorem{phenomenon}[theorem]{Phenomenon}

\numberwithin{table}{section}

\let\qqed=\qed

\def\QED{\qqed\medskip}
\let\qed=\QED
\newcommand{\Prob} {{\mathbb P}}

\newcommand{\e}{\mathrm e}
\newcommand{\ep}{\varepsilon}

\newcommand{\Z}{\mathbb{Z}}

\def\T{\mathbb T}

\def\H{\mathbb{H}}

\def\SLEkk#1/{$\mathrm{SLE}(#1)$}
\def\SLEr#1/{$\mathrm{SLE(\kappa;#1)}$}
\def\SLEkr#1;#2/{$\mathrm{SLE(#1;#2)}$}
\def\SLEk/{\SLEkk{\kappa}/}
\def\SLEtwo/{\SLEkk2/}

\def\SLEab/{\SLEkr 4; {a/\hco-1}, {b/\hco-1}/}

\def\Ito/{It\^o}
\def \eps {\epsilon}
\def \P {\Prob}
\def\md{\mid}
\def\Bb#1#2{{\def\md{\bigm| }#1\bigl[#2\bigr]}}

\def\Pb{\Bb\P}

\def \E {{\mathbb E}}

\def\fk{\mathrm{FK}}
\def\free{\mathrm{free}}
\def\wired{\mathrm{wired}}

\def \proof {{ \medbreak \noindent {\bf Proof.} }}
\def\proofof#1{{ \medbreak \noindent {\bf Proof of #1.} }}

\def\bl{\bigl}

\def\bl{\begin{lemma}}
\def\el{\end{lemma}}
\def\bth{\begin{theorem}}
\def\eth{\end{theorem}}
\def\bc{\begin{corollary}}
\def\ec{\end{corollary}}
\def\bcj{\begin{conjecture}}
\def\ecj{\end{conjecture}}
\def\bpr{\begin{proposition}}
\def\epr{\end{proposition}}
\def\bde{\begin{definition}}
\def\ede{\end{definition}}
\newcommand{\be}{\begin{eqnarray}}
\newcommand{\ee}{\end{eqnarray}}
\newcommand{\bes}{\begin{eqnarray*}}
\newcommand{\ees}{\end{eqnarray*}}

\usepackage{mathrsfs}

\def\Spec{\mathscr{S}}

\def\CC{\mathcal C}

\def\1{1}

\def\lala(#1,#2){\lambda_{#1,#2}}

\def\Ess_#1{\E |\Spec_{f_{#1}}|}

\def\EI{\mathcal{I}} %EI for Edge Intensity
\def\cloud{\mathsf{cloud}}
\def\ni{\noindent}
\def\bi{\begin{itemize}}
\def\ei{\end{itemize}}

\def\HIC{arXiv:1202.2838}

\def\SmirnovFKI{MR2680496}
\def\SmirnovFKII{SmirnovFKII}

\def\GrimmettFK{MR2243761}

\def\GrimmettCoupling{MR1379156}
\def\Bernoullicity{MR1913108}

\def\LMMSRS1{MR1124260}

\def\PivotalMeasure{arXiv:1008.1378}
\def\GPS1{MR2736153}

\def\DPSL{DPSL}

\def\DFKSL{FKdpsl}

\def\FKpivotals{FKpivotals}
\def\SpecificHeat{SpecificHeat}

\def\Onsager{MR0010315}
\def\Kadanoff{PhysRev.188.859}
\def\FerdinandFisher{FerdinandFisher}

\def\ClementThesis{ClementThesis}
\def\BdTplane{arXiv:0902.1882}
\def\BdTperiodic{arXiv:0812.3848}
\def\HVising{arXiv:1010.0526}
\def\HVselfdual{arXiv:1006.5073}

\def\FKsokal{FKsokal}
\def\SwendsenWang{\SwendsenWang}

\def\WWIsing{MR2560997}
\def\HugoStasBuzios{arXiv:1109.1549}

\def\KestenScaling{MR88k:60174}

\def\NolinKesten{MR2438816}
\def\NolinWerner{MR2505301}

\def\SmirnovPerc{MR1851632}

\def\SmirnovWerner{MR1879816}

\def\BeffaraDim6{MR2078552}

\def\BCKS{MR1868996}

\def\SchSLE{MR1776084}

\begin{document}
\maketitle

\begin{abstract}
We study the near-critical FK-Ising model. First, a determination of the correlation length defined via crossing probabilities is provided.  Second, a phenomenon about the {near-critical} behavior of FK-Ising is highlighted, 
which is completely missing from the case of standard percolation: in any monotone coupling of FK configurations $\omega_p$  ({\em e.g.}, in the one introduced in \cite{\GrimmettCoupling}), as one raises $p$ near $p_c$, the new edges arrive in a self-organized way, so that the {correlation length} is not governed anymore by the number of {pivotal edges} at criticality.
\end{abstract}

\section{Introduction}\label{s.intro}

%In this paper, we shall discuss (and prove a few properties of) the near-critical regime in the FK-percolation (especially for $q=2$). Pivotal points (and their associated critical exponent $\xi_4=5/4$) are essential in the comprehension of these regimes in percolation. These critical exponents for FK(2) percolation and the spin-Ising model are studied in the paper \cite{\FKpivotals}. However, in FK(2) percolation we will see that, somewhat unexpectedly, they contribute to the {near-critical} regime in a very different manner than in the case of standard percolation. 

%Let us start by recalling  very briefly the models of random-cluster model and the Ising model. See \cite{\GrimmettFK} or \cite{\GrimmettGraphs} or \cite{\GiiHggM} for complete introductions.

%In this introduction, we shall briefly present the model of random-cluster model, and then we will 
%succinctly describe the main results and observations which are the content of the present note.

\paragraph{Phase transition in the random cluster model on $\mathbb Z^2$.}

The random-cluster model with parameters $p\in[0,1]$ and $q\geq 1$ \footnote{In general one could take $q>0$, but we will assume $q\geq 1$ here, in order for the 
{\bf FKG inequality} to hold, see \cite{\GrimmettFK}.} is a probability measure on subgraphs of a finite graph $G=(V,E)$, defined for all 
$\omega\subset E$ by  

\[
\phi_{p,q}(\omega)~:=~ \frac{p^{\#\; \text{open edges}} (1-p)^{\#\; \text{closed edges}} q^{\# \; \text{clusters}}}{Z_{p,q}}\,,
\]
where $Z_{p,q}$ is the normalization constant such that $\phi_{p,q}$ is a probability measure. The most classical example of the random-cluster model is bond percolation, which corresponds to the $q=1$ case. Even though the random-cluster model can be defined on any graph, we will restrict ourselves to the case of the {\em square lattice $\mathbb Z^2$}.
Infinite volume measures can be constructed using limits of the above measures along exhaustions by finite subsets (with different boundary conditions: free, wired, etc). Random-cluster models exhibit a phase transition at some critical parameter $p_c=p_c(q)$. On $\Z^2$, this value does not depend on which infinite volume limit we are using, and, as in standard percolation, below this threshold, clusters are almost surely finite, while above this threshold, there exists (a.s.) a unique infinite cluster. See Subsection~\ref{ss.basic} for details and references.

The critical parameter is known to be equal to $1/2$ for bond percolation on the square lattice. For the random-cluster model with parameter $q=2$ (also called {\em FK-Ising}), $p_c(2) = \frac{\sqrt{2}}{1+\sqrt{2}}$ is known since Onsager \cite{\Onsager} (it is connected via the Edwards-Sokal coupling to the critical temperature of the Ising model). See also the recent \cite{\HVising} for an alternative proof of this fact. More recently,  the general equality $p_c(q)=\frac{\sqrt{q}}{1+\sqrt{q}}$ was proved for every $q\geq 1$ in \cite{\HVselfdual}.

Mathematicians and physicists are interested in the properties of the critical phase $(p,q)=(p_c(q),q)$ itself. Very little is known for general values of $q$, and only $q=1$ and $q=2$ are understood in a satisfactory fashion. For these two cases, the phase transition is known to be of second order ({\em i.e.} continuous): at $p_c$,  almost surely there are only finite clusters. For $q=2$, Smirnov proved the {\bf conformal invariance} of certain macroscopic observables \cite{\SmirnovFKI,\SmirnovFKII}, which, combined with the Schramm-L\"owner Evolution \cite{\SchSLE}, can be used to construct {\bf continuum scaling limits} that retain the macroscopic cluster structure \cite{CDHKS12}. For $q=1$, site percolation on the triangular lattice (a related model) has been proved to be conformally invariant as well \cite{\SmirnovPerc}.

Beyond the understanding of the critical phase, the principal goal of statistical physics is to study the phase transition itself, and in particular the behavior of macroscopic properties (for instance, the density of the infinite-cluster for $p>p_c(q)$). It is possible to relate the critical regime to these thermodynamical properties via the study of the so-called {\bf near-critical regime}. This regime was investigated in \cite{\KestenScaling} in the case of percolation. Many works followed afterward, culminating in a rather good understanding of {dynamical} and {near-critical} phenomena in standard percolation \cite{%\SchrammSteif,
\GPS1,%\NolinWerner, 
\PivotalMeasure,\DPSL}. The goal of this article is to discuss the near-critical regime in the random-cluster case, and more precisely in the FK-Ising case.

\paragraph{Correlation length of the FK-Ising model.}

The near-critical regime is the study of the random-cluster model of edge-parameter $p$ in the box of size $L$ when $(p,L)$ goes to $(p_c,\infty)$. Note that, on the one hand, if $p$ goes to $p_c$ very quickly the configuration in the box of size $L$ will look critical. On the other hand, if $p$ goes to $p_c$ (from above) too slowly, the random-cluster model will look supercritical.  The typical scale $L=L(p)$ separating these two regimes is called the {\bf correlation length} (or {\bf characteristic length}). 

Let us first define the correlation length formally in the case of percolation ($q=1$). Consider \emph{rectangles} $R$ of the form $[0,n]\times[0,m]$ for $n, m >0$, and translations of them. We denote by $\mathcal C_v(R)$ the event that there exists a \emph{vertical crossing} in $R$, a path from the bottom side $[0,n] \times \{0\}$ to the top side $[0,n] \times \{m\}$ that consists only of open edges. The classical Russo-Seymour-Welsh theorem shows that in the case of critical percolation, crossing probabilities of rectangles of bounded aspect ratio remain bounded away from 0 and 1. A natural way of describing the picture as being critical is to check that crossing probabilities are neither near 0 nor near 1. Mathematically, we thus define the correlation length for every $p<p_c=1/2$ and $\ep>0$ as
$$L_\ep(p)~:=~\inf\big\{n>0\,:\,\P_{p}\big(\mathcal C_v([0,n]^2)\big)\le \ep\big\}\,,$$
and, when $p>p_c=1/2$, as $L_{\ep}(p):=L_\ep(1-p)$, where $1-p$ is the dual edge-weight. The dependence on $\ep$ is not relevant since $L_{\ep}(p)/L_{\ep'}(p)$ remains bounded away from 0 and 1 uniformly in $p$, as shown in \cite{\KestenScaling,\NolinKesten}.
% that for any $\ep>0$, $$L_{1/4}(p)\asymp L_\ep(p)\,,$$
%where $\asymp$ means that there exist constants $0<A_\ep,B_\ep<\infty$ such that for any $p$,
%$$A_\ep L_{1/4}(p)~\le~L_\ep(p)~\le~ B_\ep L_{1/4}(p)\,.$$
The correlation length was shown to behave like $|p-p_c(1)|^{-4/3+o(1)}$ in the case of percolation \cite{\SmirnovWerner}.

Let us generalize the definition of {correlation length} to the case of FK-Ising percolation (i.e., $q=2$). Since the Russo-Seymour-Welsh theorem has been 
generalized to this case in \cite{DHN10} (see Theorem \ref{RSW critical} in the present text), it is natural to characterize 
the {\it near-critical regime} once again by the fact that crossing probabilities remain strictly between 0 and 1.
An important difference from the $q=1$ case is that one has to take into account the effect of {\it boundary conditions} (see Subsection~\ref{ss.basic} for precise definitions):

\begin{definition}[Correlation length]
Fix $q=2$ and $\rho>0$. For any $n\geq 1$, let $R_n$ be the rectangle $[0,n] \times [0, \rho n]$. 

If $p<p_c(2)$, for every $\eps>0$ and boundary condition $\xi$, define 

\[
L_{\rho, \eps}^\xi(p)~:=~\inf\left\{n>0\,:\,\phi_{p,2,R_n}^\xi \big(\mathcal C_v(R_n)\big)\le \ep\right\}\,,
\]
where $\phi_{R,p,2}^{\xi}$ denotes the random-cluster measure on $R$ with parameters $(p,2)$ and boundary condition $\xi$. If $p>p_c(2)$, define similarly
\[
L_{\rho, \eps}^\xi(p)~:=~\inf\left\{n>0\,:\,\phi_{p,2,R_n}^\xi \big(\mathcal C_v(R_n)\big)\ge 1-\ep\right\}\,.
\]
\end{definition}

\paragraph{Statements of the results.} The first result of this paper is a determination of the behavior of $L^\xi_{\rho,\ep}(p)$ when $p$ goes to $p_c$ for $q=2$. 

\begin{theorem}\label{thm:correlation length}
Fix $q=2$. For every $\ep ,\rho >0$, there is a constant $c=c(\eps,\rho)$ such that 
%CHANGED SQRT
\[
c \frac {1}{|p-p_c|}  \le L_{\rho,\eps}^\xi(p)  \le c^{-1}  \frac{1}{|p-p_c|} \log \frac {1}{|p-p_c|}
\]
for all $p\neq p_c$, whatever the choice of the boundary condition $\xi$ is. Moreover, for $\rho>1$, the logarithmic factor can be omitted. %ADDED RHO
\end{theorem}

Note that the left-hand side of the previous theorem has the following reformulation, which we state as a theorem (this result is interesting on its own, since it provides estimates on crossing probabilities which are uniform in boundary conditions away from the critical point):

\begin{theorem}[RSW-type crossing bounds] \label{RSW offcritical}
For $\lambda>0$ and $\rho\geq 1$, there exist two constants $0 <c_- \leq c_+<1$ such that for any rectangle $R$ with side lengths $n$ and $m \in [ \frac 1\rho n, \rho n]$, any $p\in[p_c-\frac \lambda n,p_c+\frac \lambda n]$ and any boundary condition $\xi$, one has
$$c_- \leq \phi_{R,p,2}^{\xi}(\mathcal C_v(R))\leq c_+\,.$$
\end{theorem}

The main ingredient of the proof of the latter theorem is Smirnov's fermionic observable. This observable is defined in Dobrushin domains (with a free and a wired boundary arc), and is a key ingredient in the proof of conformal invariance at criticality. Nevertheless, its importance goes far beyond that proof, in particular because it can be related to connectivity properties of the FK-Ising model. We study its properties away from the critical point (developing further the methods of \cite{\HVising}), and estimate its behavior near the free arc of Dobrushin domains. It implies estimates on the probability for sites of the free arc to be connected to the wired arc. This, as in \cite{DHN10}, allows us to perform a second-moment estimate on the number of connections between sites of the free arc and the wired arc, therefore implying crossing probabilities in Dobrushin domains. All that remains is to get rid of the Dobrushin boundary conditions (which is not as simple as one might hope; in particular, harder than in \cite{DHN10}), in order to obtain crossing probabilities with free boundary conditions.

Using conformal invariance techniques, Chelkak, Hongler and Izyurov have recently proved that
\begin{equation}\label{1-arm}
\phi_{p_c,q=2}  \big( 0 \leftrightarrow \partial[-n,n]^2\big) ~ \sim ~ C\, n^{-1/8}\,,
\end{equation}
together with the appropriate (conformally invariant) version for general domains  \cite{\HIC}. Together with Theorem~\ref{thm:correlation length}, this will imply:

\begin{theorem}\label{th.theta}
Assuming \eqref{1-arm}, there exists a constant $c>0$ such that if $p>p_c(2)$,
%CHANGED SQRT
\[
\phi_{p,2}(0 \leftrightarrow \infty) \ge c\, \left( \frac{|p-p_c|}{\log 1/|p-p_c| }\right)^{1/8} \,.
\]
\end{theorem}

Such results are of course not completely new. Estimates on the behavior of the correlation length were already available. One obtains \cite{McWu} that at inverse temperature $\beta$ for the Ising model (the corresponding claim for FK-Ising is easy to derive)
$$\lim_{n\rightarrow \infty} -\frac{1}{n} \ln \langle
    \sigma(0)\sigma(n e_1) \rangle = \mathrm{arcsinh}
  \sqrt{(\sinh 2\beta +\sinh ^{-1}
  2\beta)^2-1},$$
  which behaves like $|\beta_c-\beta|^{-1}$ as $\beta\nearrow \beta_c$ (the formula for the correlation length was not presented like this in \cite{McWu} and the reformulation above is extracted from \cite{Mes}). 
  %There have been several alternative derivations of this result since then, including the approach of \cite{\HVising} 
The above formula is also proved in \cite{\HVising}, using Smirnov's fermonic observable (and these techniques  will be crucial in our proof, too). Alternative approaches in the physics literature for the correlation length exponent being 1 appeared in \cite{\FerdinandFisher} and \cite{\Kadanoff}.
 
Even though very precise estimates on the correlation length were known, they do not imply directly Theorem~\ref{thm:correlation length}. Indeed, in the aforementioned works, the notion of correlation length is different (it is the inverse of the speed of exponential decay in the disordered regime) and less suitable for the study of the geometric properties of the near-critical regime than the one in our Theorem~\ref{thm:correlation length}. Its equivalence with our notion is {\bf not known rigorously}. 

The critical 1-arm exponent (\ref{1-arm}) and the off-critical result 
$$
c|p-p_c|^{1/8}\le\phi_{p,2} \big( 0 \leftrightarrow \infty \big) \le c^{-1}|p-p_c|^{1/8}
 \qquad (\text{as }p \searrow p_c)
$$
for a positive constant $c>0$ go back to Onsager \cite{\Onsager} and Yang \cite{Yan} and can be derived from the corresponding results about the Ising model. 
Still, we believe our proof of Theorem~\ref{th.theta} to be of some value, since the result of \cite{\HIC} and the techniques in this paper extend to isoradial graphs (with additional work) while Onsager's technology is restricted to the square lattice.

Finally, we should mention that spin correlations have been computed in the near-critical regime. The asymptotics of the two-point correlation function was computed in \cite{McW1,McW2,Tra}. More generally, the asymptotics of $n$-point correlation functions in the full plane were computed in the near-critical regime in \cite{Pal}. %We also refer to \cite{McWu} for a complete description of the state of the art in 1973. 

The proofs of Theorems \ref{thm:correlation length},  \ref{RSW offcritical} and \ref{th.theta} are presented in Section~\ref{sec:proofs}.
%\begin{theorem}[Onsager, \cite{\Onsager}]\label{th.Onsager1}
%For every $x,y \in \Z^2$, 
%\begin{align*}
%\langle \sigma_x , \sigma_y \rangle_{\beta_c}  = \phi_{p_c,q=2}  \big( x \leftrightarrow y\big)  \asymp |x-y|^{-1/4}\,.
%\end{align*}
%\end{theorem} 

\paragraph{The random-cluster model through its phase transition.}

The previous way to look at the near-critical regime may seem slightly artificial. It is more natural to study the random-cluster model through its phase transition by constructing a monotone coupling of random-cluster models with fixed cluster-weight $q\ge1$. Then, properties of the monotone coupling (which can be thought of as a dynamics following the evolution of $p$ between 0 and 1) near $p_c$ will describe the near-critical regime.

In the case of standard bond percolation ($q=1$), such a monotone coupling simply consists of i.i.d.\ Uniform$[0,1]$ labels on the edges, and a percolation configuration $\omega_p$ of density $p$ is the set of bonds with labels at most $p$. The {\bf near-critical window} in percolation was studied by Kesten in
\cite{\KestenScaling}, then by \cite{\BCKS,\NolinKesten, \NolinWerner, \PivotalMeasure, \DPSL}. It turns out that its size is governed by the expected number of {\bf macroscopically pivotal edges} at criticality, {\em i.e.}, edges having four alternating (between dual and primal) open paths starting there and going to a macroscopic distance. Indeed, 
let $\alpha_4(n)$ be the probability at criticality that an edge has four alternating paths going to distance $n$. Getting from $\omega_{p_c}$ to $\omega_{p_c+\Delta p}$ in the box of size $L$, the system is moving out of stationarity, and roughly $L^2 \Delta p$ edges are switched from closed to open (we are assuming $\Delta p >0$). The expected number of opened edges that were closed macroscopic pivotals in the {\em initial} configuration (preventing macroscopic open paths) is about $L^2 \alpha_4(L)  \Delta p$. Now, if $L^2\alpha_4(L) \Delta p \gg 1$, it is not very hard to show that many of these initial macroscopic pivotals have become open with good probability (not only their expected number is large), and this implies that the window of size $L$ has become well-connected. That is, we have left the near-critical regime.

On the other hand, the regime $L^2\alpha_4(L) \Delta p \ll 1$ is more difficult to understand. The number of initial macroscopic pivotals that have switched is small, but maybe many new pivotals have appeared during the dynamics, which could have switched then, establishing macroscopic open connections. To formulate the same issue from a different point of view, if the dynamics, instead of switching always from closed to open, was symmetric dynamical percolation (where each edge is flipping its state according to an independent exponential clock of rate one), then the system would be critical all the time, and, using Fubini and the linearity of expectation, the expected number of {macroscopic pivotal switches} (i.e., flips of edges that are macroscopically pivotal {\em at the moment of the flip}) in time $\Delta p$ would be  $L^2\alpha_4(L) \Delta p$. If this expectation is small, then the probability of having any macroscopic pivotal switches is also small, hence the system indeed has not changed macroscopically. However, the asymmetric near-critical dynamics is slowly moving out of criticality, which could have an effect on the number of pivotals, speeding up changes.

Nevertheless, Kesten proved the following {\bf near-critical stability} result \cite{\KestenScaling}: as long as $L^2 \alpha_4(L)  \Delta p = O(1)$, there are not many more pivotal points in $\omega_p$ than at criticality, hence, despite the monotonicity of the dynamics, changes do not speed up significantly compared to symmetric dynamical percolation, and hence the macroscopic geometry starts changing significantly only when $L^2 \alpha_4(L)  \Delta p$ becomes of order 1. Thus the following scaling relation holds: 
\begin{equation}\label{eq:KestenScaling} L_\ep(p)^2\alpha_4(L_\ep(p))(p-p_c)\asymp 1,\end{equation}
where $\asymp$ means that the quantity remains bounded away from 0 and $\infty$ uniformly in $p$.

The proof in \cite{\KestenScaling} employs Russo's influence formula and differential inequalities.
There is a related but more geometric approach in \cite{\DPSL}, which relies much less on the independence in percolation, hence will be crucial in understanding the case of FK-percolation.
% which provides some additional insight into the mechanism governing the near-critical regime, 
Namely, \cite{\DPSL} proves the following {\bf dynamical stability} result about symmetric dynamical percolation: as long as $L^2\alpha_4(L) \Delta t = O(1)$, in order to describe the macroscopic structure of $\omega_{\Delta t}$, it is enough to know the macroscopic structure of $\omega_0$ and to follow the flips experienced by all initial macroscopic pivotals. In other words, there are no cascades of information from small to much larger scales, i.e., edges initially pivotal only on a ``mesoscopic'' scale are unlikely to have a macroscopic impact in the dynamics within the given time frame. This is proved using induction, with a careful summation over all possible ways in which ``smaller'' pivotals can make a big difference.
%For this summation argument, it is of course essential (and this will be important later) that edges are switching independently from each other. 
%Now,  the monotone coupling of percolation can be interpreted in a dynamical way: starting from critical percolation at time zero, as time goes on, whenever the exponential rate 1 clock of a bond rings, the bond turns to open. The argument in \cite{\DPSL} can be extended to this framework, and it provides an alternative proof of \eqref{eq:KestenScaling}.
Now, the same argument applies to the asymmetric dynamics (the monotone coupling), and gives the near-critical stability that we stated above: in order to change the macroscopic connectivity structure, initial macroscopic pivotals need to be flipped. 

The  main principle we shall extract from this discussion is that in the case of percolation ($q=1$), due to near-critical stability, the near-critical behavior is governed  by the number of pivotal points at criticality. 

To our knowledge, it has been widely believed in the community that basically the same mechanism should hold in the case of random-cluster models. Namely, once we understand the geometry of the set of pivotal points at criticality, we may readily deduce information on the dynamical and near-critical behavior. However, this turns out to be right only for the dynamical behavior, not for the near-critical regime.

Let us consider the case of the FK-Ising. It is shown in \cite{\FKpivotals} that the critical FK-Ising probability $\alpha^{\fk}_4(n)$ for a site to be pivotal behaves like $n^{-35/24+o(1)}$ when $n$ goes to infinity. If pivotal points were governing the near-critical regime, the correlation length should satisfy
\begin{equation}\label{e.wrong}
\big(L_\ep^{\fk}(p)\big)^2\alpha_4(L_\ep^{\fk}(p))|p-p_c|~\asymp~ 1\,,
\end{equation}
which would give
\begin{equation}\label{eq:discrepancy}
L_\ep^{\fk}(p)~=~|p-p_c(2)|^{-\frac{24}{13}+o(1)}~\gg~|p-p_c(2)|^{-1},
\end{equation}
contradicting Theorem~\ref{thm:correlation length}.

In fact, monotone couplings for the random-cluster model with $q>1$ behave differently from the one in percolation. First, there is a basic phenomenon in the $\fk(p,q)$ models for $q\geq 2$ that is very relevant to the above discussion: the difference between the average densities of edges for $p=p_c(q)+\Delta p$ and $p=p_c(q)$ is not proportional to $\Delta p$, but larger than that, with an exponent given by the so-called {\bf specific heat} of the model. (We will discuss this in more detail in Subsection~\ref{s.SH}.) A first guess could be that the discrepancy in \eqref{eq:discrepancy} is a result of the fact that $\Delta p$ is not the density of the new edges arriving, and this should have been taken into account in the computation using the pivotal exponent. However, this is only partially right: the specific heat exponent itself is not large enough to account for this discrepancy (in fact, for $q=2$ it equals 0 --- there is only a logarithmic blow-up). The main reason for the discrepancy is that  a {\bf self-organizational mechanism} kicks in, as follows.

In standard percolation, the monotone coupling is just the asymmetric version of dynamical percolation, hence,
on the way from $\omega_{p_c}$ to $\omega_{p_c+\Delta p}$, new edges arrive in a ``Poissonian'' way.  Similarly to dynamical percolation, there is a natural dynamics with the random-cluster model $\fk(p,q)$ as stationary distribution, called the {\bf heat-bath dynamics} or Sweeny algorithm (see \cite{\FKsokal}, for instance): edges have independent exponential clocks, and when the clock of $e=\langle x,y\rangle$ rings, the state of $e$ is updated according to the $\fk(p,q)$ measure conditioned on the rest of the configuration $\omega$, of which the only relevant information is whether $x$ and $y$ are connected in $\omega\setminus\{e\}$. Now, in \cite{\DFKSL}, the analogue of the above-mentioned dynamical stability result of \cite{\DPSL} is proved for the critical FK-Ising model $\fk(p_c(2),2)$. If there was {\em any} monotone coupling of the near-critical $\fk(p,2)$ models in which new edges arrived one-by-one, in a Markovian way, with clock rates and resampling probabilities bounded away from 0, then the same argument would apply, and near-critical stability would hold, proving (\ref{e.wrong}). Since this prediction is wrong, how can monotone couplings look like? We know of one such coupling, due to Grimmett  \cite{\GrimmettCoupling}, which we will describe in detail in Section~\ref{s.NC}. This coupling is in fact Markovian in $p$, and given any two edges that are closed in the configuration $\omega_p$, their probabilities to be open in $\omega_{p+\Delta p}$ are comparable to each other. There is only one way how the above proof strategy of near-critical stability can break down for this coupling: there must be {\bf atoms in the measure of labels}, i.e., values of $p$ at which not just one edge appears but many, and these edges can ``arrange between each other'' where to arrive without violating Markovianity.
%the Poissonian picture can be ruled out:
%This is no longer the case with FK-Ising: 
This way, it becomes possible for the arriving edges to prefer ``strategic'' locations, creating and then opening new pivotals at large scales, thereby speeding up the dynamics compared to what could be guessed from the number of pivotals at criticality. In other words, near $p_c$,  the arriving edges depend in a very sensitive way on the current configuration. This balance between the current configuration and the conditional law of the arriving edges is representative of a self-organized mechanism.

\begin{phenomenon}\label{pheno2}
The correlation length in FK-Ising is much smaller than what the intuition coming from standard percolation ($q=1$) would predict. As one raises the parameter $p$, the supercritical regime appears ``faster'' than what would be dictated simply by the number of pivotal edges at criticality: new edges arrive in a very non-uniform manner,
%Pivotal edges are still an important aspect of the mechanism that  governs the near-critical behavior, yet, as we shall discuss more in Sections~\ref{s.NC} to~\ref{s.clouds}, 
and a {\bf self-organized near-criticality} appears. 
\end{phenomenon}

In Subsection~\ref{ss.coupling}, we introduce Grimmett's monotone coupling of the $\fk(p,q)$ configurations as $p$ varies from 0 to 1 and $q\geq 1$ is fixed. In Subsection~\ref{s.SH}, we explain heuristically why the specific heat effect on the edge intensity is not strong enough to make the correlation length what it is actually, unless there is some self-organized behavior. In Subsection~\ref{ss.clouds1}, we will present a concrete way in which self-organization works in Grimmett's coupling, by proving that edges appear simultaneously in {\bf clouds}. However, most (if not all) of the self-organized scheme remains to be understood. We therefore included a series of open questions about these clouds in Subsection~\ref{ss.clouds2}. 

\paragraph{Influences against pivotal points.} In Section~\ref{ss.influence}, we present another point of view that might explain the discrepancy (\ref{eq:discrepancy}): for $q>1$, Russo's formula used in Kesten's proof has to be modified. Namely, the influence of an edge on the event that a box of size $n$ is crossed does not coincide with the probability for that edge to be pivotal, as it was the case for $q=1$. Let us define the critical exponent $\iota(q)$ by assuming that the above influence of an edge behaves like $n^{-\iota(q)}$ at criticality. Kesten's scaling relation for $q=1$ was $(2-\xi_4(q))\nu(q)=1$, where $\xi_4(q)$ and $\nu(q)$ are the critical exponents of the pivotal event and the correlation length, respectively, coming from \eqref{eq:KestenScaling}. This will remain valid for $q>1$ only if the critical exponent $\xi_4(q)$ is replaced by the exponent $\iota(q)$ governing the behavior of the influence. The fact that $\xi_4(q)\ne \iota(q)$ seems to be new. However, this more analytic way of handling the problem quickly becomes intractable, for instance when trying to prove near-critical stability.
%a dynamical stability in the near-critical regime.

\section{Proofs of Theorems \ref{thm:correlation length},  \ref{RSW offcritical} and \ref{th.theta}}\label{sec:proofs}

\subsection{Basic properties of random-cluster models}\label{ss.basic}

The random-cluster measure can be defined on any graph. However, we 
restrict ourselves to the standard square lattice $\Z^2$. 
With a tiny abuse of notation, we will use $V(\Z^2)$ or just $\Z^2$ for the set of sites, and $E(\Z^2)$ for the set of bonds. In this paper,
$G$ will always denote a connected subgraph of $\Z^2$, \emph{i.e.},
a subset of vertices together with all the bonds
between them. We denote by $\partial G$ the (inner) boundary of $G$,
\emph{i.e.}, the set of sites of $G$ linked by a bond to a site of
$\mathbb{Z}^2\setminus G$.

A \emph{configuration} $\omega$ on $G$ is a random subgraph of $G$,
having the same sites and a subset of its bonds. We will call the bonds
belonging to $\omega$ \emph{open}, the others \emph{closed}. Two sites
$a$ and $b$ are said to be \emph{connected} (denoted by
$a\leftrightarrow b$), if there is an \emph{open path} --- a path
composed of open bonds only --- connecting them. The (maximal) connected
components will be called \emph{clusters}. More generally, we extend
this definition and notation to sets in a straightforward way.

A \emph{boundary condition} $\xi$ is a partition of $\partial G$. We
denote by $\omega \cup \xi$ the graph obtained from the configuration
$\omega$ by identifying (or \emph{wiring}) the vertices in $\xi$ that
belong to the same class of $\xi$. A boundary condition encodes the way
in which sites are connected outside of $G$. Alternatively, one can see
it as a collection of \emph{abstract bonds} connecting the vertices in
each of the classes to each other. We still denote by $\omega \cup \xi$
the graph obtained by adding the new bonds in $\xi$ to the configuration
$\omega$, since this will not lead to confusion.  Let $o(\omega)$ (resp.
$c(\omega)$) denote the number of open (resp.\ closed) bonds of $\omega$
and $k(\omega,\xi)$ the number of connected components of
$\omega\cup\xi$. The probability measure $\phi^{\xi}_{G,p,q}$ of the
random-cluster model on a \emph{finite} subgraph $G$ with parameters
$p\in[0,1]$ and $q\in(0,\infty)$ and boundary conditions $\xi$ is defined
by
\begin{equation}
  \label{probconf}
  \phi_{G,p,q}^{\xi} (\left\{\omega\right\}) := \frac
  {p^{o(\omega)}(1-p)^{c(\omega)}q^{k(\omega,\xi)}} {Z_{G,p,q}^{\xi}},
\end{equation}
for any subgraph $\omega$ of $G$, where $Z_{G,p,q}^{\xi}$ is a
normalizing constant known as the \emph{partition function}. When there
is no possible confusion, we will drop the reference to parameters in
the notation.

\paragraph{The domain Markov property.}

One can encode, using an appropriate boundary condition $\xi$, the
influence of the configuration outside a sub-graph on the measure within
it.  Consider a graph $G=(V,E)$ and a random-cluster measure
$\phi^{\psi}_{G,p,q}$ on it.  For $F\subset E$, consider $G'$ with $F$
as the set of edges and the endpoints of it as the set of sites. Then,
the restriction to $G'$ of \smash{$\phi^{\psi}_{G,p,q}$} conditioned to
match some configuration $\omega$ outside $G'$ is exactly
\smash{$\phi_{G,p,q'}^{\xi}$}, where $\xi$ describes the connections
inherited from $\omega\cup \psi$ (two sites are wired if they are
connected by a path in $\omega\cup\psi$ outside $G'$ --- see Lemma 4.13 in
\cite{\GrimmettFK}). This property is the direct analog of the DLR
conditions for spin systems.

\paragraph{Comparison of boundary conditions when $q\geq1$.}

An event is called \emph{increasing} if it is preserved by addition of
open edges. When $q\geq 1$, the model satisfies the FKG-inequality, or is \emph{positively associated} 
(see Lemma 4.14 in \cite{\GrimmettFK}), which has the following consequence:
for any boundary conditions $\psi\leq \xi$ (meaning that $\psi$ is finer
than $\xi$, or in other words, that there are fewer connections in
$\psi$ than in $\xi$), we have
\begin{equation}
  \label{comparison_between_boundary_conditions}
  \phi^{\psi}_{G,p,q}(A)\leq \phi^{\xi}_{G,p,q}(A)
\end{equation}
for any increasing event $A$. This last property, combined with the
domain Markov property, provides a powerful tool to study the decorrelation between events.

\paragraph{Examples of boundary conditions: free, wired, Dobrushin.}

Three boundary conditions play a special role in the study of
random-cluster models:
\begin{itemize}
\item The \emph{wired} boundary conditions, denoted by
$\phi_{G,p,q}^1$, is specified by the fact that all the vertices on the
boundary are pairwise connected.
\item The \emph{free} boundary conditions,
denoted by \smash{$\phi_{G,p,q}^0$}, is specified by the absence of wirings between boundary sites. 
\end{itemize}
These boundary conditions are extremal for stochastic
ordering, since any boundary condition is smaller (resp.\ greater) than
the wired (resp.\ free) boundary conditions.
\begin{itemize}
\item The {\em Dobrushin} boundary conditions: Assume now that $\partial G$ is a
self-avoiding polygon in $\mathbb L$, let $a$ and $b$ be two sites
of $\partial G$. The triple $(G,a,b)$ is called a \emph{Dobrushin
  domain}. Orienting its boundary counterclockwise defines two oriented
boundary arcs $\partial_{ab}$ and $\partial_{ba}$; the Dobrushin boundary conditions are
defined to be free on $\partial_{ab}$ (there are no wirings between boundary
sites) and wired on $\partial_{ba}$ (all the boundary sites are pairwise
connected). These arcs are referred to as the \emph{free arc} and the
\emph{wired arc}, respectively. The measure associated to these boundary
conditions will be denoted by $\phi_{G,p,q}^{a,b}$ or simply
$\phi_{G}^{a,b}$.
\end{itemize}

\paragraph{Infinite-volume measures and the definition of the critical 
point.}

The domain Markov property and comparison between boundary conditions
allow us to define infinite-volume measures. Indeed, one can consider a
sequence of measures on boxes of increasing sizes with free boundary
conditions. This sequence is increasing in the sense of stochastic
domination, which implies that it converges weakly to a limiting
measure, called the random-cluster measure on $\Z^2$ with free
boundary condition (and denoted by $\phi_{p,q}^0$). This classic construction
can be performed with many other sequences of measures, defining several
\emph{a priori} different infinite-volume measures on $\Z^2$. For
instance, one can define the random-cluster measure $\phi_{p,q}^1$ with wired
boundary condition, by considering the decreasing sequence of
random-cluster measures on finite boxes with wired boundary condition.

On $\Z^d$, for a given $q\geq 1$, it is known that uniqueness of the infinite-volume measure can fail only for $p$ 
in a countable set $\mathcal{D}_q$, see Theorem 4.60 of 
\cite{\GrimmettFK}.  Since all limit measures are sandwiched between $\phi_{p,q}^0$ and $\phi_{p,q}^1$ w.r.t.~stochastic domination, the countability of $\mathcal{D}_q$ implies that there exists a \emph{critical point} $p_c$ 
such that for {\em any} infinite-volume measure with $p<p_c$ (resp.\ $p>p_c$), 
there is almost surely no infinite component of connected sites (resp.\ 
at least one infinite component).

\paragraph{Planar duality.} In two dimensions, a random-cluster measure on a subgraph $G$ of $\mathbb{Z}^2$ with free boundary conditions can be associated with a dual measure in a natural way. First define the \emph{dual lattice} $(\mathbb{Z}^2)^*$, obtained by putting a vertex at the center of each face of $\mathbb{Z}^2$, and by putting edges between nearest neighbors. The \emph{dual graph} $G^*$ of a finite graph $G$ is given by the sites of $(\mathbb{Z}^2)^*$ associated with the faces adjacent to an edge of $G$. The edges of $G^*$ are the edges of $(\mathbb{Z}^2)^*$ that connect two of its sites -- note that any edge of $G^*$ corresponds to an edge of $G$.

A dual model can be constructed on the dual graph as follows: for a percolation configuration $\omega$, each edge of $G^*$ is \emph{dual-open} (or simply open), resp. \emph{dual-closed}, if the corresponding edge of $G$ is closed, resp. open. If the primal model is a random-cluster model with parameters $(p,q)$, then it follows from Euler's formula (relating the number of vertices, edges, faces, and components of a planar graph) that the dual model is again a random-cluster model, with parameters ($p^*,q^*)$ -- in general, one must be careful about the boundary conditions. For instance, on a graph $G$, the random-cluster measure $\phi_{G,p,q}^{0}$ is dual to the measure $\phi_{G^*,p^*,q^*}^{1}$, where $(p^*,q^*)$ satisfies
\begin{equation*}
\frac{p p^*}{(1-p) (1-p^*)} = q \quad\text{and}\quad q^*=q.
\end{equation*}
Similarly, the dual of Dobrushin boundary conditions are Dobrushin boundary conditions themselves. 

The critical point $p_c(q)$ of the model is the self-dual point $p_{\textrm{sd}}(q)$ for which $p = p^*$
(this has been recently proved in \cite{\HVselfdual}), whose value can be derived:
\begin{equation*}
p_{\textrm{sd}}(q)=\frac{\sqrt{q}}{1+\sqrt{q}}.
\end{equation*}

In the following, we need to consider connections in the dual model. Two sites $x$ and $y$ of $G^*$ are said to be \emph{dual-connected} if there exists a connected path of open dual-edges between them. Similarly to the primal model, we define \emph{dual clusters} as maximal connected components for dual-connectivity.

\paragraph{FK-Ising model: crossing probabilities at criticality.}
For the value $q=2$ of the parameter, the random-cluster model is related to the Ising model. In this case, the random-cluster model is now well-understood. The uniqueness of the infinite volume limit for all $p$ is known since Onsager; see \cite{\WWIsing} for a short and elegant proof (or \cite[Proposition 3.10]{\HugoStasBuzios} for a version of Werner's proof in English). The value $p_c= p_{\mathrm{sd}}$ is implied by the computation by Kaufman and Onsager of the partition function of the Ising model, and an alternative proof has been proposed recently by Beffara and Duminil-Copin \cite{\HVising}. Moreover, in \cite{\SmirnovFKI}, Smirnov proved conformal invariance of this model at the self-dual point $p_{\mathrm{sd}}$. As mentioned in the Introduction, criticality can sometimes be characterized by the fact that crossing probabilities are bounded uniformly away from 0 and 1. This fact was proved in \cite{DHN10} (our result is an extension of this one away from criticality):

\begin{theorem}[RSW-type crossing bounds, \cite{DHN10}]\label{RSW critical}
Let $0 < \rho_1 < \rho_2$. There exist two constants $0 <c_- \leq c_+<1$ (depending only on $\rho_1$ and $\rho_2$) such that for any rectangle $R$ with side lengths $n$ and $m \in [ \rho_1 n, \rho_2 n]$ (\emph{i.e.} with aspect ratio bounded away from $0$ and $\infty$ by $\rho_1$ and $\rho_2$), one has
$$c_- \leq \phi_{R,p_c,2}^{\xi}(\mathcal C_v(R))\leq c_+$$
for \emph{any} boundary conditions $\xi$.
\end{theorem}

\textbf{In the rest of this section, we consider only random-cluster models on the
  two-dimensional square lattice with parameter $q=2$}, hence we drop the
dependency on $q$ in the notation. In this case, the model is called FK-Ising model. In addition, a point will be identified with its complex coordinate.

%On the one hand, the key property of the fermionic observable at criticality is its discrete holomorphicity (and thus harmonic). It allowed Smirnov to identify the scaling limit of such observables to be a conformally covariant object, namely the solution to a certain Riemann-Hilbert boundary problem. On the other hand, it has been proved in \cite{BDC11} that Smirnov's observable is still quite integrable away from the critical point. More precisely, it still satisfies local relations, which do not correspond to discrete holomorphic functions but to massive harmonic ones. When $p$ is close enough to $p_c$, it is still possible to consider the function as being roughly harmonic. The difficulty is that the boundary conditions correspond to a discretization of Riemann-Hilbert problems, and are quite messy. Smirnov overcomes this difficulty by introducing a new function $H$, that we unfortunately cannot use away from criticality. We still manage to study the observable away from the critical point in specific domains, where we can harness a representation via massive random-walks.

\subsection{Connectivity probabilities and the fermionic observable}
\label{sec:definition}

The random-cluster with cluster-weight $q=2$ is a model with long-range dependence. In particular, boundary conditions play a crucial role in connectivity probabilities. While general percolation arguments are sometimes sufficient to estimate connectivity probabilities in the bulk \cite{\HVselfdual}, there are very few possibilities to control probabilities in the presence of boundary conditions. We thus need a new argument to control these crossing probabilities. 

When $q=2$, Smirnov's fermionic observable provides us with a powerful tool to study such probabilities. In the next paragraph, we introduce the loop representation of the random-cluster model and we define Smirnov's observable. In the next one, we remind several properties of this observable at and away from the critical point. We list them without proof, since they are already presented in various places.

\paragraph{The medial lattice and the loop representation.}

Let $G$ be a finite subgraph of $\Z^2$ together with Dobrushin boundary condition given by the boundary points $a$ and $b$. 
Let $G^*$ be the dual graph, with the natural definition that respects the boundary condition, see the left side of Figure~\ref{fig:medial_lattice}. Declare \emph{black} the sites of $G$ and \emph{white} the sites of $G^*$. Replace every site with a colored
diamond, as in the right side of Figure~\ref{fig:medial_lattice}. The \emph{medial graph}
$G_{\diamond}=(V_{\diamond},E_{\diamond})$ is defined as follows: $E_\diamond$ is the set of
diamond sides which belong to both a black and a white diamond;
$V_\diamond$ is the set of all the endpoints of the edges in
$E_\diamond$. We obtain a subgraph of a rotated (and rescaled) version
of the usual square lattice. We give $G_\diamond$ an additional
structure as an oriented graph by orienting its edges clockwise around
white faces.

\begin{figure}[ht]
  \begin{center}
    \includegraphics[width=.6\hsize]{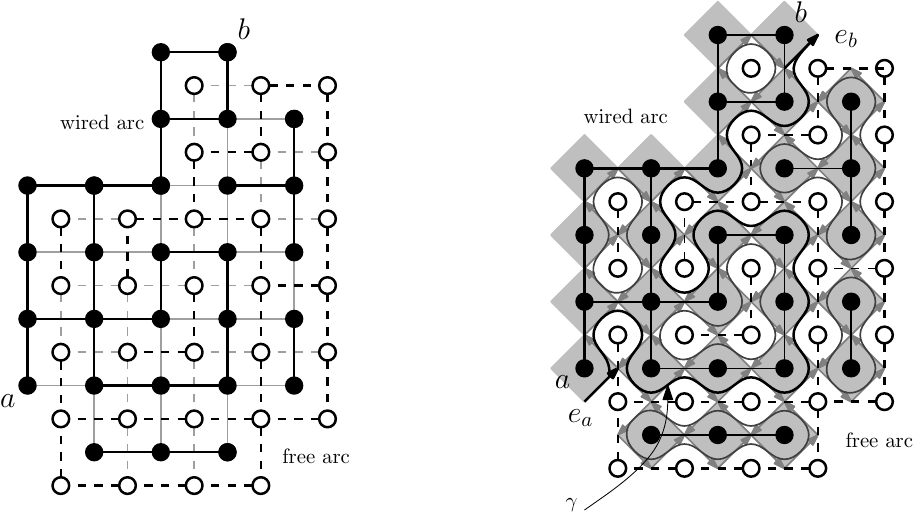}
  \end{center}
  \caption{\textbf{Left:} A graph $G$ with Dobrushin boundary conditions, and its dual $G^*$. The black
    (resp.\ white) sites are the sites of $G$ (resp. $G^*$). The open
    bonds of $G$ (resp. $G^*$) are represented by solid (resp.\ dashed)
    black bonds. \textbf{Right:} Construction of the medial lattice and
    the loop representation: the loops are interfaces between primal and
    dual clusters.}
  \label{fig:medial_lattice}
\end{figure}

The random-cluster measure on $(G,a,b)$ with Dobrushin boundary conditions has a 
rather convenient representation in this setting. Consider a 
configuration $\omega$. It defines clusters in $G$ and dual clusters in 
$G^\star$. Through every vertex of the medial graph passes either an open 
bond of $G$ or a dual open bond of $G^\star$, hence there is a unique way to 
draw Eulerian (\emph{i.e.}, using every edge exactly once) loops on the 
medial lattice --- \emph{interfaces}, separating clusters from dual 
clusters. Namely, a loop arriving at a vertex of the medial lattice 
always makes a $\pi/2$ turn so as not to cross the open or dual open 
bond through this vertex, see Figure~\ref{fig:medial_lattice}. Besides 
loops, the configuration contains a single curve joining the vertices 
adjacent to $a$ and $b$, which are the only vertices in $V_{\diamond}$ 
with three adjacent edges.  This curve is called the 
\emph{exploration path} and is denoted by $\gamma$. It corresponds 
to the interface between the cluster connected to the wired arc and the 
dual cluster connected to the free arc.

The first and last edge of $\gamma$ are denoted by $e_a$ and $e_b$, respectively. 
More generally, $e_c$ denotes the medial edge pointing north-east and bordering the diamond associated to $c$. 

This construction gives a bijection between random-cluster configurations on $G$ and
Eulerian loop configurations on $G_{\diamond}$. The probability measure
can be nicely rewritten (using Euler's formula) in terms of the loop
picture: $$\phi_{G}^0 ({\omega}) = \frac {x(p)^{\# \; \text{open
      bonds}} \sqrt{2}^{\# \; \text{loops}}} {\tilde{Z}(p,G)}, \quad
\text{where} \quad x(p) := \frac p {(1-p) \sqrt{2}}$$ and
$\tilde{Z}(p,G)$ is a normalizing constant. Notice that $p=p_c$ if
and only if $x(p)=1$.  This bijection is called the \emph{loop
  representation} of the random-cluster model. The orientation of the
medial graph gives a natural orientation to the interfaces in the loop
representation.

\paragraph{The fermionic observable.}

Fix a Dobrushin domain $(G,a,b)$. Following~\cite{\SmirnovFKI}, we now
define an observable $F$ on the edges of its medial graph, \emph{i.e.}\
a function $F : E_{\diamond} \to \mathbb{C}$. Roughly speaking, $F$ is a
modification of the probability that the exploration path passes through
a given edge.

First, the \emph{winding}
$\text{W}_{\Gamma}(z,z')$ of a curve $\Gamma$ between two edges $z$ and
$z'$ of the medial graph is the total rotation (in radians and oriented
counter-clockwise) that the curve makes from the mid-point of edge $z$
to that of edge $z'$.  We define the
observable $F=F_p$ for any edge $e\in E_{\diamond}$ as
\begin{equation}
  \label{defF}
  F(e) := \phi_{p,2,G}^{a,b} \left({\rm e}^{\frac{\rm i}{2}
  \text{W}_{\gamma}(e,e_b)} \mathbbm{1}_{e\in \gamma}\right),
\end{equation}
where $\gamma$ is the exploration interface from $a$ to $b$.

\paragraph{Relation with connectivity probabilities.} 

As was mentioned earlier, the fermionic observable is related to connectivity properties of the model via the following fact:
\begin{lemma}[Equation (14) in \cite{\SmirnovFKI}, Lemma 2.2 in \cite{\HVising}]
  \label{boundary}
  Let $u\in G$ be a site next to the free arc, and $e$ be a side of the black
  diamond associated to $u$ which borders a white diamond of the boundary. Then,
  \begin{equation}
    |F(e)|=\phi_{p,2,G}^{a,b}(u\leftrightarrow \text{wired arc}).
  \end{equation}
\end{lemma}

\paragraph{Integrability relations of the fermionic observable.}

A lot of information has been gathered on the fermionic observable during the last few years. In particular, it satisfies local relations that allow to determine its scaling limit.

\begin{proposition}[Lemma 2.3 of \cite{\HVising}]
  \label{integrability,relation_around_a_vertex}
  Consider a medial vertex $v$ in $G^\diamond\setminus \partial G^\diamond$. We index the two edges pointing towards $v$ by $A$ and $C$, and the two others by $B$ and $D$, such that  the alphabetical order is clockwise oriented. Then,  
  \begin{equation}
    \label{rel_vertex}
    F(A)-F(C) ~=~\e^{i\alpha}~i[F(B)-F(D)].
  \end{equation}
 where
\begin{equation}\label{eq:alphafromp}
  {\rm e}^{{\rm i}\alpha}:=  \frac{{\rm e}^{{\rm i}\pi/4} +x}{{\rm
      e}^{{\rm i}\pi/4}x+ 1}.
\end{equation} 
\end{proposition}

Moreover, the complex argument modulo $\pi$ of $F$ at any edge $e$ is determined by the direction of $e$. More precisely, if $e$ points in the same direction as $e_b$, then $F$ is real. Similarly, $F$ belongs to $\e^{-i\pi/4}\mathbb R$ (resp. $i\mathbb R$, $\e^{i\pi/4}\mathbb R$) if $e$ makes an angle of $\pi/2$ (resp. $\pi$, $3\pi/2$) with the edge $e_b$, see Fig.~\ref{fig:relation}. Knowing the complex argument modulo $\pi$, together with (\ref{rel_vertex}), allow one to express the value of the observable at one edge $e$ in terms of the values at two edges incident to one of the endpoints of $e$. This important fact was used extensively in \cite{\SmirnovFKI} and in any following work since. For instance, it implies that $F$ is determined by the relations \eqref{rel_vertex} and the fact that $F(e_b)=1$ and $F(e_a)=\e^{iW_{\Gamma}(e_a,e_b)/2}$, where $\Gamma$ is any path from $e_a$ to $e_b$ staying in $G^\diamond$. In addition, these relations have a very special form. In particular, it has been proved for $x=1$ in \cite{\SmirnovFKI} and then extended in \cite{\HVising} (Lemma~4.4) to general values of $x$ that $F$ is massive harmonic inside the domain:

\begin{proposition}\label{mass}
Let $p\in(0,1)$ and $X$ with four neighbors in $G\setminus \partial G$, we have 
\begin{equation}
\Delta_pF(e_X)~=~0,
\end{equation}
where the operator $\Delta_p$ is defined by
\begin{equation}
\Delta_pg(e_X)~:=~\frac{\cos [2\alpha]}{4}\left(\sum_{Y\sim X} g(e_Y)\right)~-~g(e_X).
\end{equation}
\end{proposition}

\begin{figure}[ht]
  \begin{center}
    \includegraphics[width=.5\hsize]{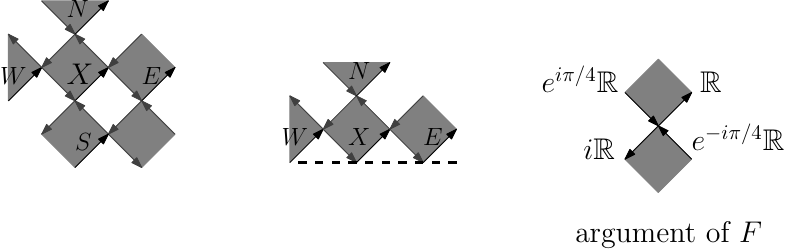}
  \end{center}
  \caption{\textbf{Left:} An edge inside the domain: it has four edges oriented the same way at distance two. \textbf{Right:} An edge on the free arc with the associated indexation.}
  \label{fig:relation}
\end{figure}
Observe that $\alpha(p)=0$ if and only if $p=p_c$. In this case, the observable is discrete harmonic inside the domain. As mentioned before, this is one of the main ingredients of Smirnov's proof of conformal invariance: when properly rescaled, the observable converges to a harmonic map. Boundary conditions for $F$ correspond to discretizations of the Riemann-Hilbert problem. These boundary conditions are quite complicated to study at a discrete level, and Smirnov used a discrete primitive $H$ of (the imaginary part of) $F^2$ to handle them. The function $H$ was then solving a discretized Dirichlet problem. 

In particular, the use of $H$ made the estimation of $F$ on the free arc possible. More precisely, $F$ was related to the square root of modified harmonic measures (see Proposition~3.2 of \cite{DHN10}). This fact was crucial in the proof of Theorem~\ref{RSW critical} \cite{DHN10}. Without entering into details, let us say that in Dobrushin ``domains'' $(G,a,b)$, the probability at criticality for a site $x$ on the free arc to be connected to the wired arc (which is the modulus of the observable, thanks to Lemma~\ref{boundary}) is on the order of the square-root of the harmonic measure of the wired arc seen from $x$. Equivalently, for a dual site $u$ on the wired arc, the probability of being dual-connected to the free arc is of order of the square-root of the harmonic measure of the free arc seen from $u$. We refer to \cite{DHN10} for additional details on these facts.

We will be using this fact for two nice infinite Dobrushin domains:
\begin{itemize}
\item The infinite strip $S_n=\mathbb Z\times[0,n]$ of height $n$. Denote by $\phi_{S_n,p}^{-\infty,\infty}$ the random-cluster measure with parameter $p$, free boundary conditions on the bottom and wired boundary conditions on the top. The probability at criticality for a dual-site on the top to be dual-connected to the free arc is of order $1/\sqrt n$ (since the harmonic measure of the free arc is $1/n$, via the Gambler's ruin).
\item The upper half-plane $\mathbb H$. Denote by $\phi_{\mathbb H,p}^{0,\infty}$ the random-cluster measure with parameter $p$, free boundary conditions on $\Z_+=\{0,1,\cdots\}$ and wired boundary conditions on $\Z_-=\{\cdots,-2,-1,0\}$. The probability at criticality for the dual site adjacent to $-n$ to be dual-connected to the free arc is of order $1/\sqrt n$ for the same reason as for the strip.
\end{itemize}

When $p\le p_c$, the fermionic observable $F_p$ can be defined in these two domains even though they are infinite (see \cite{\HVising}). In the strip, $\gamma$ goes from $-\infty$ to $\infty$, while in $\H$, it goes from $0$ to $\infty$.
One should be careful about the definition of the winding since $e_b$ does not make sense: the winding is fixed to be equal to 0 on edges of the free arc pointing north-east. Since the observable in infinite volume is the limit of observables in finite volume, it still satisfies the properties of the previous section.

\begin{figure}[ht]
  \begin{center}
    \includegraphics[width=.3\hsize]{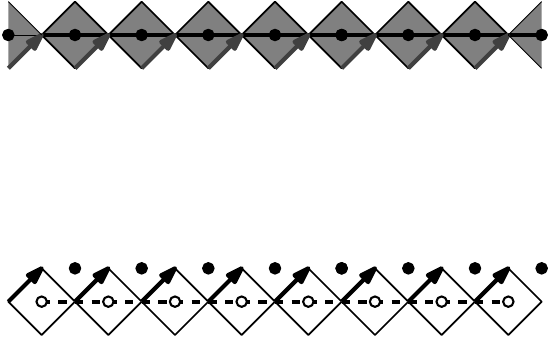}\quad\quad\includegraphics[width=.5\hsize]{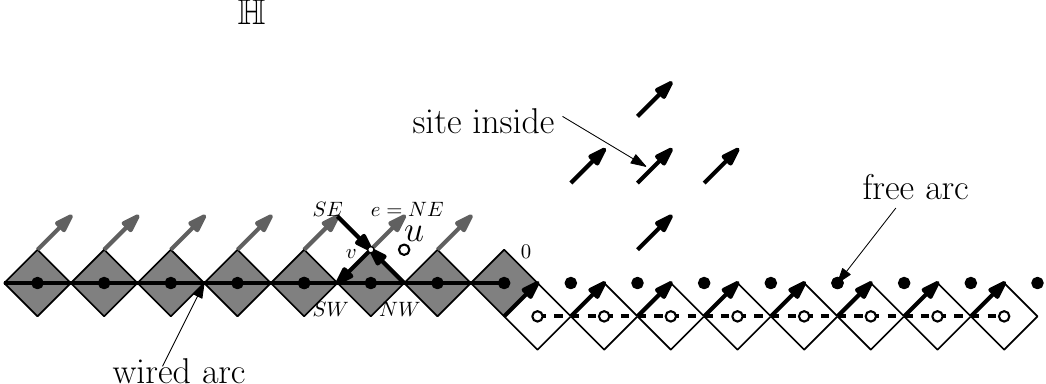}
  \end{center}
  \caption{\textbf{Left:} The strip. \textbf{Right:} The upper half-plane. In this case, $\tau$ is the hitting time of grey edges.}
  \label{fig:domains}
\end{figure}

%\begin{proposition}[uniform comparability]\label{uniform comparability}
%Let $(\Omega, a, b)$ be a discrete Dobrushin domain, and let $e$ be a medial edge adjacent to the free arc. Let $B=B(e)$ be the black face bordered by $e$ and $W=W(e)$ one white face adjacent to $e$ and the free arc but not on the free arc. Then we have
%\begin{equation}
%\sqrt{H_\circ(B)} ~\leq~ |F (e)|~\le \sqrt{H_\bullet (W)},
%\end{equation}
%where $H_\circ(B)$ (resp. $H_\bullet(W)$) is the probability for a modified random-walk starting from $B$ (reps. $W$) to exit the domain through the wired arc.
%\end{proposition}

\subsection{Integrability relations of the fermionic observable away from criticality}

Away from the critical point, the primitive $H$ is not available anymore. Nevertheless, $F_p$ is massive harmonic inside the domain. In fact, $F_p$ satisfies very explicit relations on the free arc of the domain, as well. (On the wired arc, such relations are not available.) Indeed, Case 3 of Lemma~4.4 of \cite{\HVising} says that
\small
\begin{align}
\Delta_p F_p(e_X) &= \frac{\cos 2\alpha}{2(1+\cos (\pi/4-\alpha))}[F_p(e_W)+F_p(e_N)]~+~\frac{\cos (\pi/4+\alpha)}{1+\cos (\pi/4-\alpha)}F_p(e_E)~-~F_p(e_X)\nonumber\\
&=~0\label{e.massivebdry}
\end{align}
\normalsize
for $X$ on the free arc (excluding 0 in the case of the upper half-plane). When $p=p_c$, the sum of the coefficients on the right equals 0, which means that the observable has an interpretation in terms of reflected random walks. This relates to the discretization of the Riemann-Hilbert boundary problem, and it  provides an alternative strategy to handle the scaling limit of the observable. %Unfortunately, the previous relation is available only for nice pieces of boundary, and the choice of the strip and the upper half-plane were made in such a way that this relation can be used.

Away from criticality, we can also interpret these relations in terms of a random process. Define the Markov process with
  generator $\Delta_p$, which one can interpret as the random walk of a massive particle. We write this process $(X_n^{(p)},m_n^{(p)})$ where $X_n^{(p)}$ is a
  random walk with jump probabilities defined in terms of $\Delta_p$ ---
  the proportionality between jump probabilities is the same as the
  proportionality between coefficients --- and $m_n^{(p)}$ is the mass
  associated to this random walk.  The law of the random walk starting
  at an edge $x$ is denoted $\mathbb{P}^x_p$.  In order to simplify the notation, we drop the dependency in $p$ in $(X^{(p)}_n,m^{(p)}_n)$ and simply write $(X_n,m_n)$. Note that the mass of the walk
  decays by a factor $\cos 2\alpha$ at each step inside the domain, and by some other factor on the free arc.
  
Define $\tau$ to be the hitting time of the wired arc, more precisely, of the set $\partial$ of medial edges pointing north-east and having one end-point on the wired arc (the grey edges in Fig.~\ref{fig:domains}). The fact that $\Delta_pF_p=0$ for every edge $x\notin\partial$ implies for any $t\geq 0$ that
\begin{equation}\label{aaaaaaaa}F_p(x)= \mathbb E^x_p[F_p(X_{t\wedge \tau})m_{t\wedge\tau}].\end{equation}
Since $m_\tau \le 1$, $F_p(X_{t\wedge \tau})m_{t\wedge \tau}$ is uniformly integrable and \eqref{aaaaaaaa} can be improved to
  \begin{equation}\label{bbbbbbbb}F_p(x)=\mathbb E^x_p[F_p(X_\tau)m_\tau].\end{equation}
 This will be the principal tool in our study.

\begin{proposition}\label{low strip}Let $\lambda>0$. There exists $C_1=C_1(\lambda)$ such that for every $n>0$ and $p_c\ge p>p_c-\frac{\lambda}{n}$,
\begin{equation}\label{low}\phi_{S_n,p}^{-\infty,\infty}(0\leftrightarrow in+\Z)~\ge~\frac{C_1}{\sqrt n}.\end{equation}
There exists $C_2=C_2>0$ such that for every $n>0$ and $p<p_c-\frac{C_2 \log n}{n}$,
%CHANGED SQRT
\begin{equation}\label{upp}\phi_{S_n,p}^{-\infty,\infty}(0\leftrightarrow in+\Z)~\leq~\frac{C_2}{n^4}.
\end{equation}
\end{proposition}

\begin{proof}
In both cases, we study the probability for a point on the free arc to be connected to the wired arc. In particular, Lemma~\ref{boundary} implies that the quantities on the left of \eqref{low} and \eqref{upp} are equal to $|F(e_0)|$ (or $F(e_0)$ in this case, since the winding is fixed on the boundary). Moreover, \eqref{bbbbbbbb} allows us to write
\begin{align*}\phi_{S_n,p}^{-\infty,\infty}(0\leftrightarrow in+\Z)~&=~F(e_0)~=~\mathbb E^0_p[F_p(X_\tau)m_\tau]\end{align*}

Let us first deal with~\eqref{low}. Note that $F_p(X_\tau)=F_p(in)=F_{p^*}(e_0)$, where $p^*$ is the dual parameter. Hence
\begin{align*}\phi_{S_n,p}^{-\infty,\infty}(0\leftrightarrow in+\Z)~&=~\mathbb E^0_p[ F_{p^*}(e_0)\, m_\tau]\\
&=~\phi_{S_n,p^*}^{-\infty,\infty}(0 \leftrightarrow in+\Z)\,\mathbb E^0_p[m_\tau]\,,
\end{align*}
using Lemma~\ref{boundary} again. Since $p>p_c-\frac{\lambda}{n}$, $\cos 2\alpha=m_1^{(p)}$ is larger than  $1-c(\lambda/n)^2$ for some $c>0$. Using the upper tail of $\tau/n^2$, we deduce that 
$$\mathbb E^x[m_\tau]\ge C$$
for some  $C=C(\lambda)$. In addition to this, 
\begin{equation*}\phi_{S_n,p^*}^{-\infty,\infty}(0 \leftrightarrow in+\Z)~\ge~\frac C{\sqrt n},
\end{equation*}
where we used the standard estimate of the probability at criticality, together with $p^*>p_c$. The two inequalities together yield \eqref{low}.

Let us now turn to \eqref{upp}. When $p<p_c-C \log n/n$, we use the expansion of $\alpha$ near $p_c$ and $\cos 2\alpha\le 1-c (\log n)^2 /n^2$ (for some constant $c=c(C)$) to deduce using standard large deviations arguments about random walks:
%CHANGED SQRT
\begin{equation*}\phi_{S_n,p}^{-\infty,\infty}(0\leftrightarrow in+\Z)~\le~\mathbb E^0_p[m_\tau]~\le~\mathbb E^0_p\big[\big(1-c_2(\log n)^2/n^2\big)^\tau\big]\leq C_2 \,n^{-4}\end{equation*}
%CHANGED SQRT
%for $c'=c'(C)$. In order to conclude, $c'$ can be chosen larger than 4 by tuning $C$.\QED
for some well-chosen constant $C_2=C_2(C)$. \QED
\end{proof}

The previous proof of \eqref{low} was based on a comparison with the estimates at criticality: when $p>p_c-\lambda/n$ the connection probabilities are of the same order as the critical ones. We push this reasoning further in the following proposition.

\begin{proposition}\label{estimate in H}For any $\lambda>0$, there exists $C_3=C_3(\lambda)>0$ such that for every $n>0$ and $p>p_c-\frac{\lambda}{n}$,
\begin{equation}
\phi_{\mathbb H,p}^{0,\infty}(n\leftrightarrow \Z_-)~\ge~\frac{C_3}{\sqrt n}.
\end{equation}
\end{proposition}

Let us first prove an easy yet slightly technical result. It should be compared to Lemma~\ref{boundary}.

\begin{lemma}\label{extension boundary}
Let $u$ be a dual vertex adjacent to the wired arc of $\mathbb H$, 
$$F_p(e_u)\asymp\phi_{\mathbb H,p}^{0,\infty}(u\stackrel{\star}{\leftrightarrow} \Z_+),$$
where $e_u$ is the edge pointing north-east and adjacent to $u$, and $\asymp$ means that the ratio is uniformly bounded away from 0 and $\infty$.
\end{lemma}

\begin{proof}
If $v$ is the vertex of the medial lattice on the left of $u$, relation \eqref{rel_vertex} around $v$ gives
  $F(NW)+F(SE)=\e^{i\alpha}(F(NE)+F(SW))$, where edges are indexed with respect to the direction {\bf they are pointing to} (see Fig.~\ref{fig:domains}). Since we know the complex argument modulo $\pi$ of the observable, we can project the relation on $\e^{-i\pi/4}\mathbb R$, to find
 $$\e^{i\pi/4}F(NW)-\cos (\pi/4-\alpha)iF(SW)~=~\cos(\pi/4+\alpha)~F(NE).$$
Now, the argument of the observable at $NW$ and $SW$ is in fact determined, since the winding on the boundary is deterministic (it equals $-\pi/2$ for $NW$, and $-\pi$ for $SW$). Therefore, Lemma~\ref{boundary} implies 
\begin{align*}
\e^{i\pi/4}F(NW)~&=~|F(NW)|~=~\phi_{\mathbb H,p}^{0,\infty}(u\stackrel{\star}{\leftrightarrow} \Z_+)\\
iF(SW)~&=~|F(SW)|~=~\phi_{\mathbb H,p}^{0,\infty}(u-1\stackrel{\star}{\leftrightarrow} \Z_+).
\end{align*}
So, using the fact that $NE=e_u$, we get
  $$\cos (\pi/4+\alpha)F(e_u)~=~\phi_{\mathbb H,p}^{0,\infty}(u\stackrel{\star}{\leftrightarrow} \Z_+)
-\cos(\pi/4-\alpha)\phi_{\mathbb H,p}^{0,\infty}(u-1\stackrel{\star}{\leftrightarrow} \Z_+).$$
Now, $\phi_{\mathbb H,p}^{0,\infty}(u-1\stackrel{\star}{\leftrightarrow} \Z_+)\le \phi_{\mathbb H,p}^{0,\infty}(u\stackrel{\star}{\leftrightarrow} \Z_+)$ thanks to the comparison between boundary conditions. We deduce 
$$\frac{1-\cos(\pi/4-\alpha)}{\cos(\pi/4+\alpha)}~\phi_{\mathbb H,p}^{0,\infty}(u\stackrel{\star}{\leftrightarrow} \Z_+)~\le~F(e_u)~\le~\frac{1}{\cos(\pi/4+\alpha)}~\phi_{\mathbb H,p}^{0,\infty}(u\stackrel{\star}{\leftrightarrow} \Z_+)$$
which is the claim. \QED

%Using the fact that $NE=e_u$, we deduce
%  $$\cos (\pi/4+\alpha)F(e_u)~=~\phi_{\mathbb H,p}^{0,\infty}(u\stackrel{\star}{\leftrightarrow} \Z_+)
%-\cos(\pi/4-\alpha)\phi_{\mathbb H,p}^{0,\infty}(u+1\stackrel{\star}{\leftrightarrow} \Z_+).$$
%Now, $\phi_{\mathbb H,p}^{0,\infty}(u\stackrel{\star}{\leftrightarrow} \Z_+)\le \phi_{\mathbb H,p}^{0,\infty}(u+1\stackrel{\star}{\leftrightarrow} \Z_+)$ thanks to the comparison between boundary conditions. We deduce 
%$$\frac{1-\cos(\pi/4-\alpha)}{\cos(\pi/4+\alpha)}~\phi_{\mathbb H,p}^{0,\infty}(u\stackrel{\star}{\leftrightarrow} \Z_+)~\le~F(e_u)~\le~\frac{1}{\cos(\pi/4+\alpha)}~\phi_{\mathbb H,p}^{0,\infty}(u\stackrel{\star}{\leftrightarrow} \Z_+)$$
%which is the claim. \QED
\end{proof}

We are now in a position to prove the proposition. 

\proofof{Proposition~\ref{estimate in H}} Fix $n>0$ and $p\ge p_c-\frac{\lambda}{n}$ and denote the fermionic observable in $(\mathbb H,0,\infty)$ by $F_p$. Lemma~\ref{extension boundary} implies 
\begin{equation}\label{b}F_p(n)=\mathbb E^n_p[F_p(X_\tau)m_\tau]\asymp \mathbb E^n_p[\phi_{\mathbb H,p}^{0,\infty}(X_\tau\stackrel{\star}{\leftrightarrow} \Z_+)m_\tau].\end{equation}
We know that 
$$\phi_{\mathbb H,p}^{0,\infty}(u\stackrel{\star}{\leftrightarrow} \Z_+)\ge C_3/\sqrt {|u|},$$
hence \eqref{b} implies
\begin{equation}\label{lower Fp}F_p(n)~\ge~C_4\mathbb E^n_p[~\phi_{\mathbb H,p}^{0,\infty}(X_\tau\stackrel{\star}{\leftrightarrow} \Z_+)~m_\tau]~\ge ~C_4C_3~\mathbb E^n_p[~m^p_\tau/\sqrt{|X_\tau|}]\,,\end{equation}
with two universal constants $C_3,C_4>0$. 

Therefore, it is sufficient to prove that $|X_\tau|$ is not larger than $n$ and that $m_\tau$ is larger than some constant $\varepsilon$ with probability bounded away from 0 uniformly in $n$. The second condition can be replaced by the event $\tau\le n^2$ for instance.

The walk $X_t$ away from the real axis is just simple random walk, while on the free arc it has some outwards drift. So, it is sufficient to prove that $X_t$ exits $[0,2n]\times[0,n]$ through $[0,2n]\times\{n\}$ in fewer than $n^2/2$ steps with probability larger than some constant $c>0$ not depending on $n$. Indeed, it has then a uniformly positive probability to exit the domain in fewer than $n^2/2$ additional steps and to satisfy $|X_\tau|\le n$.

Consider $(X_t)_{t\le n^2/2}=(A_t,B_t)_{t\le n^2/2}$ conditioned on the event that $(X_t)_{t\le n^2/2}$ visits the free arc fewer than $n$ times. The probability that the first coordinate is less than $n$ for every $t\le n^2/2$ is bounded away from 0 uniformly in $n$ (since the number of visits of $(X_t)$ to the free arc is less than $n$, $(A_t)$ can be compared to a symmetric random walk with a deterministic drift of order $r n$ for $r<1$). Now, conditioned on the visits of $(X_t)$ to the free arc, $(A_t)$ and $(B_t)$ are independent. Thus, $(B_t)$ is a random walk reflected at the origin conditioned on the fact that it does not visit 0 more than $n$ times. In time $n^2/2$, it reaches height $n$ with probability bounded away from 0, uniformly in $n$. The claim follows. 
\QED

\subsection{Proof of Theorem~\ref{RSW offcritical}}

We first prove crossing probabilities in rectangles with specific boundary conditions. Then, we use these crossings to construct crossings in arbitrary rectangles with free boundary conditions.

\paragraph{Crossing in rectangles with Dobrushin boundary conditions.}

Let us first use the estimates obtained in the previous subsection to prove crossing probabilities in the strip and the half-plane. The proof follows a second moment argument. 

\begin{proposition}\label{prop:crossings}
Fix $\lambda>0$. There exists $C_6=C_6(\lambda)>0$ such that for every $n>0$ and $p > p_c-\frac\lambda n$,
$$\phi_{S_n,p}^{-\infty,\infty}([-n,n]\leftrightarrow in+\Z)\ge C_6$$
and
$$\phi_{\mathbb H,p}^{-\infty,\infty}([3n,4n]\leftrightarrow \Z_-)\ge C_6.$$
\end{proposition}

\begin{figure}[ht]
  \begin{center}
    \includegraphics[width=.6\hsize]{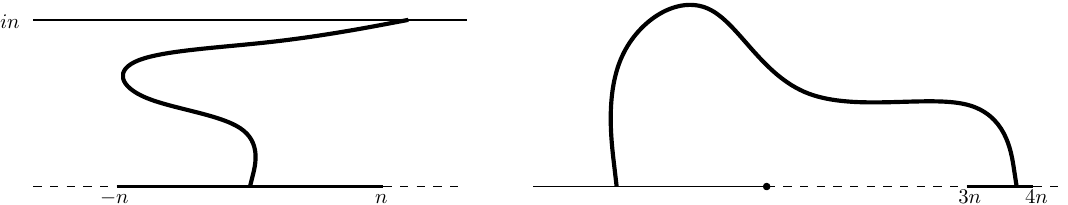}  \end{center}
  \caption{The two crossing events of Proposition~\ref{prop:crossings}.}
  \label{fig:crossings}
\end{figure}

\begin{proof}
We present the proof for $S_n$ (a similar argument works for $\mathbb H$). Let $N$ be the (random) number of sites on $[-n,n]$ which are connected by an open path to $in+\Z$. Proposition~\ref{low strip} implies that 
\begin{equation}
\phi_{S_n,p}^{-\infty,\infty}(N)~=~\sum_{x\in[-n,n]}\phi_{S_n,p}^{-\infty,\infty}(x\leftrightarrow in+\Z)~\ge~(2n+1)\frac{C_1}{\sqrt n}~\ge~2C_1\sqrt n.
\end{equation}
Moreover, for $p \leq p_c$,
$$\phi_{S_n,p}^{-\infty,\infty}(N^2)~\le~\phi_{S_n,p_c}^{-\infty,\infty}(N^2).$$
The right hand side is a quantity at the critical point and was already studied in the proof of the main theorem of \cite{DHN10} (in fact, only closely related quantities were studied, but the generalization is straightforward). In particular, it was proved in that article that $$\phi_{S_n,p_c}^{-\infty,\infty}(N^2)~\le C_6 n\,.$$
Cauchy-Schwarz thus implies that 
$$\phi_{S_n,p}^{-\infty,\infty}([-n,n]\leftrightarrow in+\Z) = \phi_{S_n,p}^{-\infty,\infty}(N>0) \ge 4C_1^2/C_6\,,$$ uniformly in $n$. For $p > p_c$ the result follows from monotonicity.
\QED
\end{proof}
\medbreak

It is now easy to reduce crossing probabilities in the strip and the half-plane to crossing probabilities in (possibly very large) rectangles. The idea is that a crossing cannot explore too much of the strip or the half-plane, since there exist slightly supercritical dual crossings preventing it.

\begin{proposition}\label{crossing weird}
Fix $\lambda>0$. There exist $C_7>0$ and $M>0$ such that for every $n>0$ and $p>p_c-\frac \lambda n$,
$$\phi_{[-Mn,Mn]\times[0,n],p}^{(i-M)n,(i+M)n}([-n,n]\leftrightarrow in+\Z)\ge C_7$$
and 
$$\phi_{[-Mn,Mn]\times[0,Mn],p}^{0,-Mn}([3n,4n]\leftrightarrow \Z_-)\ge C_7.$$
\end{proposition}
\begin{proof}
As before, we do this in the case of the strip. Fix $M$ large enough so that, at criticality, the probability that there exists a vertical dual crossing with free boundary conditions of $[n,Mn]\times[0,n]$ exceeds $1-C_6/3$ (use Theorem~\ref{RSW critical} to prove this fact). Then, with probability $C_6/3$, there will exist a crossing of $[-n,n]$ to $in+\Z$ and two dual vertical crossings in $[n,Mn]\times[0,n]$ and $[-Mn,-n]\times[0,n]$. The domain Markov property and the comparison between boundary conditions imply the result. \QED
\end{proof}
\medbreak

\paragraph{Crossing in rectangles with free boundary conditions.}

A consequence of Proposition~\ref{crossing weird} is the existence of crossings inside a box with free boundary conditions everywhere. Indeed, although the previous result only deals {\em a priori} with domains where a part of the boundary is already wired, this condition can be removed.

\begin{proposition}\label{rid of BC}
Fix $\lambda>0$. There exist $C_8,M>0$ such that for every $n>0$ and $p>p_c-\frac\lambda n$,
$$\phi_{[-Mn,Mn]\times[0,n],p}^0\big([-Mn,Mn]\times[0,n/2]\text{ is crossed vertically}\big)\ge C_8.$$
\end{proposition}

%\begin{figure}[ht]
 % \begin{center}
  %  \includegraphics[width=.7\hsize]{without_boundary}
 % \end{center}
  %\caption{The rectangles $R_k$ used in the proof.}
  %\label{fig:without boundary}
%\end{figure}

\begin{proof}
Fix $M$ so that Proposition~\ref{crossing weird} holds true. Let $A_n$ be the event that $[-Mn,Mn]\times[0,n/2]$ is crossed vertically. We have for every $n>0$,
$$\phi_{[-Mn,Mn]\times[0,n],p}^{(i-M)n,(i+M)n}(A_n)~\ge~C_7.$$
Let $B_n$ be the event that $[-Mn,Mn]\times[n/2,n]$ is dual-crossed horizontally. Theorem~\ref{RSW critical} implies that 
$$\phi_{[-Mn,Mn]\times[0,n],p}^{(i-M)n,(i+M)n}(B_n|A_n)~\ge~c$$
for some constant $c>0$, uniformly in $n$ and $p<p_c$. Now,
\begin{align*}
\phi_{[-Mn,Mn]\times[0,n],p}^0(A_n)~&\ge~\phi_{[-Mn,Mn]\times[0,n],p}^{(i-M)n,(i+M)n}(A_n|B_n)\\
&\ge~\phi_{[-Mn,Mn]\times[0,n],p}^{(i-M)n,(i+M)n}(A_n\cap B_n)\\
&=~\phi_{[-Mn,Mn]\times[0,n],p}^{(i-M)n,(i+M)n}(B_n|A_n)\cdot\phi_{[-Mn,Mn]\times[0,n],p}^{(i-M)n,(i+M)n}(A_n)\\
&\ge ~c\cdot C_7\,.
\end{align*}
\vskip -0.3in \QED
%Now, for $\ep>0$, slice $[-Mn,Mn]$ into $M/\ep$ segments of length $2\ep n$. The union over $n$ of the events that there exists a vertical crossing of $[-Mn,Mn]\times[0,\ep n]$ starting from some segment has probability $cC_7$. The union-bound implies that some of this segment is the base of a vertical crossing with probability larger than $\ep/M\, cC_7$. Now, the probability for each of these events is smaller than the probability to have a crossing of $[-2Mn,2Mn]\times[0,n/2]$ starting from $[-\ep n,\ep n]$, which implies the result. 
\end{proof}

We now prove that crossings of rectangles of any aspect ratio also exist.

\begin{lemma}\label{RSW inner}
Fix $\lambda>0$ and $\kappa>0$. There exists $C_9=C_9(\kappa,\lambda)>0$ such that for every $n$ and $p>p_c-\frac\lambda n$,
$$\phi^0_{[-n,(\kappa+1)n]\times [0,n],p}([0,\kappa n]\times [0,n]\text{ is crossed horizontally})\ge C_9.$$
\end{lemma}

\begin{figure}[ht]
  \begin{center}
    \includegraphics[width=.6\hsize]{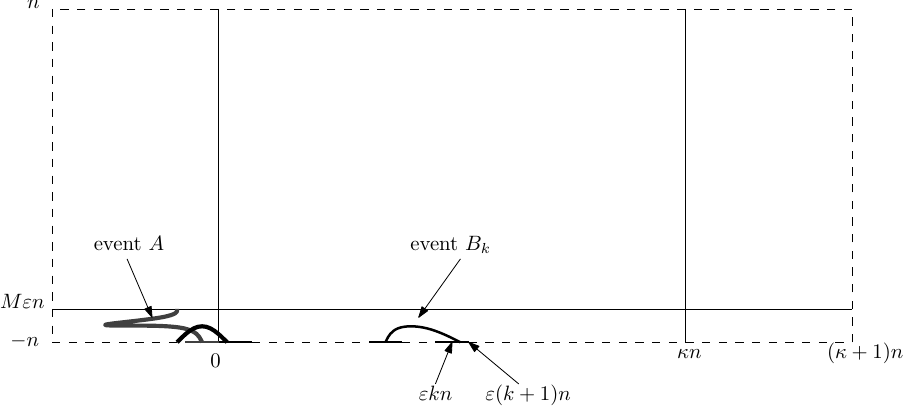}
  \end{center}
  \caption{The intersections of the events $A$ and $B_k$ create a crossing of the rectangle $[-n,n]\times[0,\kappa n]$.}
  \label{fig:crossing_long_rectangle}
\end{figure}

\begin{proof}
Fix $M=M(\lambda)$ as in Propositions~\ref{crossing weird} and~\ref{rid of BC}. Let $\ep=1/(2M)^2$. Let $A$ be the event that there exists a crossing from $[-\ep n,\ep n]$ to $iM\ep n+\Z$, and let $B_k$ be the event that there exists a path in $\Z\times[0,M\ep n]$ from $[(k+1)\ep n,(k+2)\ep n]$ to $[(k-1)\ep n,k \ep n]$. See Figure~\ref{fig:crossing_long_rectangle}. We have
\begin{align*}
\phi^0_{[-n,(\kappa+1)n]\times[0,n],p}([-n,n]\times[0,\kappa n]\text{ is crossed horizontally})&\\
&\hskip-6 cm\ge~\phi^0_{[-n,(\kappa+1)n]\times [0,n],p}\left(A\cap \bigcap_{k=0}^{\kappa/\ep-1}B_k\right)\\
&\hskip-6 cm = ~\phi^0_{[-n,(\kappa+1)n]\times[0,n],p}(A)\prod_{k=0}^{\kappa/\ep-1}\phi^0_{[-n,(\kappa+1)n]\times[0,n],p}(B_k|A,B_r,r<k)\,.\end{align*}
Furthermore, 
$$\phi^0_{[-n,(\kappa+1)n]\times [0,n],p}(A)\ge \phi^0_{[-n,n]\times[0,n/(2M)],p}(A).$$
Now, since $M\eps=1/(4M)$, the event $A$ in $[-n,n]\times[0,n/(2M)]$ corresponds to the existence of a crossing from the bottom to the middle, but with the additional constraint that it starts between $[-\ep n,\ep n]$. A union bound, comparison between boundary conditions, and Proposition~\ref{rid of BC} imply that
$$\phi^0_{[-n,n]\times[0,n/(2M)],p}(A)~\ge~\ep \phi^0_{[-n,n]\times[0,n/(2M)],p}([-n,n]\times[0,n/(4M)])~\ge~ \ep\, C_8\,.$$
Furthermore,
\begin{align*}\phi^0_{[-n,n]\times[-n,(\kappa+1)n],p}(B_k|A,B_r,r<k)~&\ge\phi^{k\ep n,\infty}_{[(k-M)\ep n,(k+M)\ep n]\times[0,M\ep n],p}(B_k)\ge C_7\,,
\end{align*}
using the comparison between boundary conditions and the half-plane case of Proposition~\ref{crossing weird}. Altogether, we obtain that
$$\phi^0_{[-n,(\kappa+1)n]\times[0,n],p}([-n,n]\times[0,\kappa n]\text{ is crossed horizontally})~\ge C_8C_7^{\kappa/\ep},$$
and the lemma is proved.
\QED
\end{proof}

\begin{figure}[ht]
  \begin{center}
    \includegraphics[width=.4\hsize]{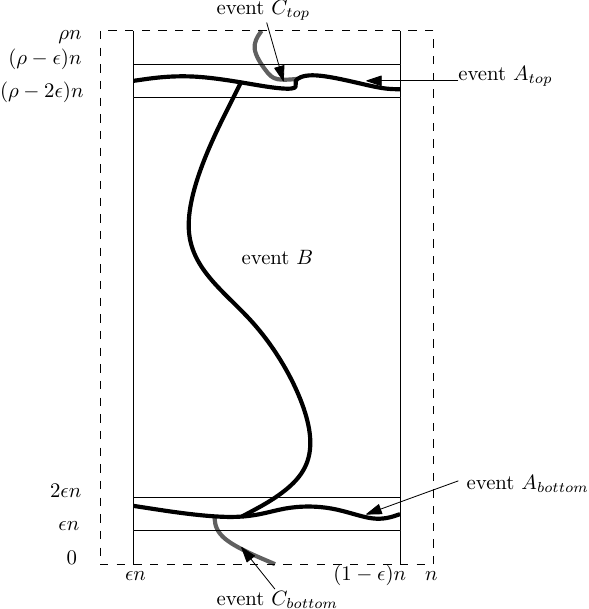}
  \end{center}
  \caption{The five events involved in the proof of Theorem~\ref{RSW offcritical}.}
  \label{fig:proof}
\end{figure}

\ni
{\bf Proof of Theorem~\ref{RSW offcritical}}
\ni
Fix $\ep<1/(4M)$. Let $A_{bottom}$ and $A_{top}$ be the events that $[\ep n,(1-\ep)n]\times[\ep n,2\ep n]$ and $[\ep n,(1-\ep)n]\times[(\rho-2\ep)n,(\rho-\ep) n]$ are crossed horizontally. Let $B$ be the event that $[\ep n,(1-\ep) n]\times[\ep n,(\rho-\ep)n]$ is crossed vertically. Let $C_{bottom}$ and $C_{top}$ be the events that $[\ep n,(1-\ep)n]\times[0,2\ep n]$ and $[\ep n,(1-\ep)n]\times[(\rho-2\ep)n,\rho n]$ are crossed vertically. See Figure~\ref{fig:proof}. By Lemma~\ref{RSW inner}, the events $A_{bottom}$, $A_{top}$ and $B$ have probability bounded away from 0 uniformly in $n$. The FKG inequality implies that their intersection also has this property. Now, conditionally on $A_{bottom}$, the event $C_{bottom}$ has probability larger than the probability that there exists a crossing in $[\ep n,(1-\ep)n]\times[0,2\ep n]$ with wired boundary condition on the top and free boundary condition on the bottom. Proposition~\ref{crossing weird} implies that this probability is larger than $C_7$ since $(1-2\ep)/(2\ep)>2M$ (the important thing is that the rectangle $[\ep n,(1-\ep)n]\times[0,2\ep n]$ is wide enough). The same reasoning can be applied to $C_{top}$, ergo the claim follows.
\QED

\subsection{Proofs of Theorems~\ref{thm:correlation length} and \ref{th.theta}}

First of all, a standard reasoning described in Step 2 of the proof of Theorem~1 of \cite{\HVising} shows that equation~\eqref{upp} implies the following lemma:

\begin{lemma}\label{lem:decay}
There exists $C_{10}>0$ such that
\begin{equation}
\phi_{p}(0\leftrightarrow \partial [-n,n]^2)~\le~C_{10}n^{-3}
\end{equation}
%CHANGED SQRT
for every $n$ large enough and every $p\le p_c-C_{10}\frac{\log n}{n}$.\qed
\end{lemma}

\proofof{Theorem~\ref{thm:correlation length}}
Fix $C_{10}>0$ as defined in Lemma~\ref{lem:decay}. Theorem~\ref{RSW offcritical} implies the lower bound trivially. For the upper bound, it suffices to show that for any $\kappa>0$,
\[
\phi^1_{[-n,n]\times [-\kappa n,\kappa n],p}([-n, n]\times [-\kappa n,\kappa n]\text{ is crossed horizontally})~\rightarrow 0
\]
%CHANGED SQRT
whenever $(n,p)\rightarrow(\infty,0)$ with $p\le p_c-C_{10}\log n/{n}$. Fix $\ep>0$ and $\kappa>0$. 

Take some $\delta>0$ to be fixed later. Let $A_n^{top}$ be the event that $[-(1-\delta)n,(1-\delta)n] \times [-\kappa n, (\kappa-2\delta)n]$ is crossed horizontally, and $A_n^{bottom}$ be the event that $[-(1-\delta)n,(1-\delta)n] \times [-(\kappa-2\delta) n, \kappa n]$ is crossed horizontally. Furthermore, let $B_n$ be the event that $[-(1-\delta)n,(1-\delta) n]\times[-(\kappa-\delta)n,(\kappa-\delta)n]$ contains a cluster of diameter $\delta n$. Notice that if the rectangle $[-n,n]\times [-\kappa n,\kappa n]$ is crossed horizontally, then $A_n^{top}$ or $A_n^{bottom}$ or $B_n$ occurs. 

Theorem~\ref{RSW critical} implies the existence of $\delta>0$ such that the probability of $A_n^{top}$ (and similarly for $A_n^{bottom}$) with wired boundary conditions is smaller than $\ep/3$ for any $p<p_{sd}$ and $n>0$. This will be our $\delta$.

Define $C_n$ to be the event that the annulus $$S_n:=[-n,n]\times [-\kappa n,\kappa n]~\setminus~[-(1-\delta)n,(1-\delta n)]\times[-(\kappa-\delta)n,(\kappa-\delta)n]$$ contains a closed circuit surrounding the inner box. Note that there exists $\eta>0$ such that 
$$\phi^1_{p,S_n}(C_n)~\ge~\eta\,,$$
thanks to Theorem~\ref{RSW critical} again. Since $C_n$ is decreasing and $B_n$ depends only on edges inside $[-(1-\delta)n,(1-\delta) n]\times[-(\kappa-\delta)n,(\kappa-\delta)n]$, we obtain
$$\phi^1_{p,[-n,n]\times [-\kappa n,\kappa n]}(C_n|B_n)~\ge~\phi_{p,S_n}^1(C_n)~\ge~\eta.$$
Therefore,
\begin{align*}
\eta \phi_{p,[-n,n]\times [-\kappa n,\kappa n]}^1(B_n)~&\le~\phi_{p,[-n,n]\times [-\kappa n,\kappa n]}^1(B_n\cap C_n)\\
~&\le~\phi_{p,[-n,n]\times [-\kappa n,\kappa n]}^1(B_n|C_n)\\
~&\le~\phi_{p,[-n,n]\times [-\kappa n,\kappa n]}^0(B_n).
\end{align*}
By a union bound, Lemma~\ref{lem:decay} and the definition of $C_{10}$ imply that
\begin{equation}\label{boom}\phi_{p,[-n,n]\times [-\kappa n,\kappa n]}^0(B_n)~\longrightarrow 0\quad\text{when }n\rightarrow 0.\end{equation} 
Therefore, $\phi_{p,[-n,n]\times [-\kappa n,\kappa n]}^1(B_n)\rightarrow 0$. 

Summarizing, each of  $A_n^{top}$, $A_n^{bottom}$ and $B_n$ has  probability less than $\ep/3$ for $n$ large enough, which  concludes the proof.

We now turn to the improvement for the case $\rho > 1$. The idea is that for $p<p_c-\frac{C_2 \log n}{n}$ the bound~(\ref{upp}) means that we are at the subcritical end of the critical window in a strong sense, from which we can deduce subcriticality in a weaker sense also at a larger $p$ value.

Let us assume that $\phi^0_p\big(\mathcal C_h([0,L]\times [0,\rho L])\big)=\eps>0$ for some $p<p_c$ and $L=\lambda/|p-p_c|$. Take $n=\frac{C}{|p-p_c|}\log\frac{1}{|p-p_c|}$, and consider horizontal crossings in the rectangle $[0,L]\times [0,n]$ with free boundary condition. By covering this rectangle by $L\times\rho L$ rectangles, with overlaps of size $L\times L$, the FKG inequality implies that 
$$\phi^0_p\big(\mathcal C_h([0,L]\times [0,n])\big)\geq \eps^{2n/L} = \eps^\frac{-2C \log |p-p_c|}{\lambda}=|p-p_c|^\frac{-2C \log \eps}{\lambda}\geq n^\frac{2C \log\eps}{\lambda}\,.$$
On the other hand, (\ref{upp}) implies that, for $C>0$ large enough in the definition of $n$, we have
$$\phi^0_p\big(\mathcal C_h([0,L]\times [0,n])\big)\leq n^{-3+o(1)}\,.$$
Comparing these two bounds gives $\lambda < 2C |\log\eps | /3$, yielding the required bound for the correlation length with free boundary conditions. The extension to wired (and hence arbitrary) boundary conditions follows exactly as in the above proof;  one should just change the definition of $B_n$ to denote the event that the rectangle $[-(1-\delta)n,(1-\delta)n] \times [-(\kappa-\delta)n,(\kappa-\delta)n]$ is crossed horizontally.
\qed

Let us now turn to the proof of Theorem~\ref{th.theta}. We have just proved that, for $\rho>0$ and $\ep>0$, there exists $c=c(\ep,\rho)$ such that for any 
%CHANGED SQRT
$n\ge \frac{c}{p_c-p} \log\frac1{p_c-p}$, 
$$\phi_{p,[0,n]\times[0,\rho n]}^1\Big(\mathcal C_h([0,n]\times[0,\rho n])\Big)~\le~ \ep.$$
The next lemma asserts that crossing probabilities in fact converge to 0 very quickly as soon as $n$ is larger than the correlation length.

\begin{lemma}\label{exponential decay}
For any $p<p_c$, there exists $L(p)$ such that
%CHANGED SQRT
$$\frac{c}{p_c-p}~\le~L(p)~\le~\frac{1}{c(p_c-p)}  \log \frac{1}{p_c-p}$$
and
$$\phi_{p,[0,2^kL(p)]\times[0,2^{k+1}L(p)]}^1 \Big( \mathcal{C}_h\big([0,2^kL(p)]\times[0,2^{k+1}L(p)]\big) \Big)~\le~\e^{-2^k}$$
for any $k\ge 0$.
\end{lemma}

\proof
For $n>0$, let 
$$u_n:=\max\left\{\phi_{p,[0,n]\times[0,2n]}^1\Big(\CC_h([0,n]\times[0,2n])\Big)\,,~\phi_{p,[0,n]^2}^1\Big(\CC_h([0,n]^2)\Big)\right\}.$$ 
We are going to show that 
\begin{equation}\label{e.u2nun}
u_{2n}\leq 25 \, u_n^2\,.
\end{equation}

First, cutting vertically the domain $[0,2n]^2$ into two rectangles, together with comparison between boundary conditions, implies that
\begin{equation}\label{e.u2nunOne}
\phi_{p,[0,2n]^2}^1\Big(\CC_h([0,2n]^2)\Big) \leq \phi_{p,[0,n]\times[0,2n]}^1\Big(\CC_h([0,n]\times[0,2n])\Big)^2 \leq u_n^2\,.
\end{equation}

Second, cutting vertically the domain $[0,2n]\times[0,4n]$ into two, together with comparison between boundary conditions again, implies that
$$
\phi_{p,[0,2n]\times[0,4n]}^1\Big(\CC_h([0,2n]\times[0,4n])\Big) \leq \phi_{p,[0,n]\times[0,4n]}^1\Big(\CC_h([0,n]\times[0,4n])\Big)^2\,.
$$

Now, consider the rectangles 
\begin{align*}
R_1~&:=~[0,n]\times[0,2n]\\
R_2~&:=~[0,n]\times[n,3n]\\
R_3~&:=~[0,n]\times[2n,4n]\\
R_4~&:=~[0,n]\times[n,2n]\\
R_5~&:=~[0,n]\times[2n,3n]
\end{align*}
These rectangles have the property that whenever $[0,n]\times[0,4n]$ is crossed horizontally, at least one of the rectangles $R_i$ is crossed (in the horizontal direction for $R_1$, $R_2$ and $R_3$, and vertically otherwise). We deduce, using the comparison between boundary conditions, that
$$
\phi_{p,[0,n]\times[0,4n]}^1\Big(\CC_h([0,n]\times[0,4n])\Big)~\le~5\, u_n\,,
$$
and hence
\begin{equation}\label{e.u2nunTwo}
\phi_{p,[0,2n]\times[0,4n]}^1\Big(\CC_h([0,2n]\times[0,4n])\Big)~\le~(5\, u_n)^2\,.
\end{equation}

Combining \eqref{e.u2nunOne} and \eqref{e.u2nunTwo}, we obtain~\eqref{e.u2nun}. Iterating that, 
we easily obtain that, for every $k\ge 0$,
\begin{align*}
25\, u_{2^kn}~&\le~ \left(25\,u_n\right)^{2^k}.
\end{align*}
By Theorem~\ref{thm:correlation length}, if $p<p_c$ and 
%CHANGED SQRT
$n\ge \frac{c^{-1}}{p_c-p} \log\frac1{p_c-p}$, where $c=\min \{c(1/100,2),c(1/100,1)\}$, then $u_n$ satisfies $$25\, u_n\le 1/\e\,.$$
Therefore, the lemma follows for 
%CHANGED SQRT
$L(p)=\frac{c^{-1}}{p_c-p} \log\frac1{p_c-p}$.
\qed

\proofof{Theorem~\ref{th.theta}} Fix $p>p_c$. Let 
$$R_k~:=~[0,L(p)2^k]~\times~[-L(p),L(p)(2^{k+1}-1)]\quad\text{ if $k$ is even},$$
and $$R_k~:=~[0,L(p)2^{k+1}]~\times~[-L(p),L(p)(2^k-1)]\quad\text{ if it is odd}.$$ Define $E_k$ to be the event that $R_k$ is crossed in the long direction. The FKG inequality implies that
\begin{align*}\phi_p^0(0\leftrightarrow \infty)~&\ge~\phi_p^0\big(0\leftrightarrow \{L(p)\}\times[-L(p),L(p)]\big)\cdot\prod_{k\ge 0}\phi_p^0(E_k)\\
&\ge~\frac14\phi_p^0\big(0\leftrightarrow \partial[-L(p),L(p)]^2\big)\cdot \prod_{k\ge 0}\left(1-\e^{-2^k}\right)\\
&\ge~c\,\big(L(p)\big)^{-1/8}\,,
\end{align*}
where $c>0$. We used Lemma~\ref{exponential decay} to get the second line, and the lower bound of Theorem~\ref{thm:correlation length} and \eqref{1-arm} to get the third inequality.
\qed

\begin{remark}
While the first of the two Kesten's scaling relations in \cite{\KestenScaling} (namely \eqref{eq:KestenScaling}) was shown in the introduction to be wrong for the FK-Ising percolation, the second scaling relation, usually written under the following form
$$\mathbb P_p(0\leftrightarrow \infty)\asymp \mathbb P_p(0\leftrightarrow \partial[-L(p),L(p)]^2),$$
 should still be valid. Indeed, the thermodynamical quantities $L(p)$ and $\mathbb P_p(0\leftrightarrow \infty)$ have their analogues in the FK-Ising case. Onsager's determination of the magnetization, together with the Edwards-Sokal coupling implies that 
\begin{align}\label{e.thetafk}
 \phi_{p,2}(0 \leftrightarrow \infty)  & \asymp |p-p_c|^{1/8} 
\end{align}
and
\begin{equation}\label{one arm FK}\phi_{p_c,2}( 0\leftrightarrow  \partial [-n,n]^2) ~ \asymp~  n^{-1/8}.\end{equation}
From these two relations, the second scaling relation (which does not harness any pivotal event) implies that the correlation length should behave as $1/|p-p_c|$ for FK-Ising, which is the right prediction. Also note that \eqref{one arm FK} has been proved using conformal invariance techniques in \cite{\HIC}. It would be interesting to make sense of the second scaling relation in the FK-Ising case in order to provide a derivation of the exponent $1/8$ for the magnetization which would be independent of Onsager's computation. Half of this is achieved by Theorem~\ref{th.theta}.
\end{remark}

\section{Monotone coupling and near-critical behavior}\label{s.NC}

In this section, we first present briefly the monotone coupling of the random-cluster model introduced by Grimmett. We use it to explain heuristically why new edges should not appear in anything like a Poissonian way. We then prove that a self-organized mechanism does exist by proving that edges arrive in {\em clouds}. The proof is very weak and provide virtually no information on these clouds, which should be crucial for further understanding of the near-critical regime. We therefore conclude this section by listing few open questions about these clouds. 

\subsection{The monotone increasing Markov process on cluster configurations}\label{ss.coupling}

We now describe briefly Grimmett's monotone coupling (see \cite{\GrimmettCoupling, \GrimmettFK,\Bernoullicity} for a detailed exposition). Let $G=(V,E)$ be a finite subgraph of $\Z^2$ and $\Omega$ be the space $[0,1]^{E}$. The goal is to find a measure $\mu=\mu_G$ on $\Omega$ in a such a way that all the ``projections'' $\omega_p(Z)$ with $Z\sim \mu$, defined by
\begin{equation}\label{e.proj}
\omega_p(Z)(e):= 1_{Z(e)\le  p}\,,\qquad  p\in[0,1],\ e\in E\,,
\end{equation}
follow the random-cluster probability measure of parameters $(p,q)$ on $\{0,1\}^{E}$ with some given boundary conditions. It turns out that it is non-trivial to construct explicitly such a measure $\mu$ (note that in contrast the existence of abstract monotone couplings follows easily from a generalized Strassen's theorem). Instead, Grimmett obtains it as the invariant measure of a natural Markov process $Z_t$ on the space $\Omega:= [0,1]^{E}$.

Let $Z_t$ be a Markov chain on $\Omega$ where labels on the edges are updated at rate one according to the conditional law defined below. For any $e=\langle x, y \rangle \in E$, let $\mathcal{D}_e\subset \{0,1\}^{E}$ be the event that there is a path of open edges in $E \setminus \{ e \}$ 
 connecting $x$ and $y$. For any $e\in E$ and any $Z\in \Omega$, define 
\[
T_e(Z):= \inf \{ p\in [0,1] \text{ s.t. } \omega_p(Z) \in \mathcal{D}_e \}\,.
\]
Let $\mathcal{U}_e$ be the random variable corresponding to the new label at $e$ and time $t$ knowing the current configuration (before the update) $Z_{t-}$ given by the law 
\begin{equation}\label{e.Ue}
\Pb{\mathcal{U}_e \le p}:= \left \lbrace \begin{array}{ll} p & \text{ if $p\geq T$} \\
 \frac{p}{p+(1-p)q} & \text{ if $p<T$\,,} \end{array}\right.
\end{equation}
where $T=T_e(Z_{t-})$. The condition $q\geq 1$ implies that this is a valid distribution function, hence we can simply define $\mathcal{U}_e$ to be a sample from this distribution. Note that $\mathcal{U}_e$ has an absolutely continuous part plus a {\bf Dirac point mass} (for $q>1$) on $T$, namely $[T - \frac{T}{T+(1-T)q}] \delta_T$. 

Constructing an infinite-volume version of the previous dynamics is not straightforward. Nevertheless, one has the following asymptotic statement from \cite{\GrimmettCoupling, \GrimmettFK}.

\begin{proposition}[Infinite Volume Limit \cite{\GrimmettCoupling}]
For each $n\geq 1$, let $\Lambda_n:= [-n,n]^d$.
Let $\xi$ be some initial configuration in $X:=[0,1]^{E(\Z^2)}$. 
For $q\geq 1$, consider the above dynamics $Z_t^{\Lambda_n}$ on $\Lambda_n$ with {\bf free} boundary conditions and which starts from the initial state $Z_0^{\Lambda_n} \equiv \xi_{\md \Lambda_n}$. Then, as $n\to \infty$, the process $(Z_t^{\Lambda_n})$ weakly converges to a {Markov} process $(Z^{\free}_t)_{t\geq 0}$ which starts from the initial configuration $Z_0^\free = \xi$.  

Furthermore, as $t \to \infty$, $Z^\free_t$ weakly converges to an invariant measure $\mu$ on $X$. 

If, in the limiting procedure, one uses {\bf wired} boundary conditions instead, one obtains at the limit a Markov process $(Z^{\wired}_t)_{t\geq 0}$. The processes $Z^\wired_t$ and $Z^\free_t$ might possibly have 
different transition kernels but they both have the same $\mu$ as the {\bf unique} invariant measure.  For $Z\sim\mu$, the projections $\omega_p(Z)$ given by~(\ref{e.proj}) have the law of $\mathrm{FK}(p,q)$, $p\in[0,1]$.
\end{proposition}

Let us now prove that Grimmett's coupling leads to a monotone increasing {\bf Markov} process (as $p$ varies) on the cluster configurations, {\em i.e.}, on the space $\{0,1\}^{E}$, providing a clear picture of the self-organization scheme near $p_c(q)$. Namely, as one raises $p$ near $p_c$, new edges arrive in a complicated fashion yet depending only on the current configuration $\omega_p$. We are not aware of a proof of this Markovianity elsewhere in the literature.

\begin{proposition}\label{pr.MonotoneMarkov}
Let $G=(V,E)$ be a finite subgraph of $\mathbb Z^2$. Let $Z$ be sampled according to the law $\mu$. Then the monotone family of projections
$(\omega_p(Z))_{0\le p \le 1}$, seen as a random process in the ``time'' variable $p$, 
is a non-decreasing {\bf inhomogeneous Markov process} on the space $\{0,1\}^{E}$. 
\end{proposition}

\proof
We wish to prove that conditioned on the projections $(\omega_u(Z))_{0\le u \le p}$, the conditional
law of the higher configurations  $(\omega_u(Z))_{p\le u \le 1}$ depends only on $\omega_p(Z)$. To achieve this, it is enough to prove Lemma~\ref{l.MonotoneMarkov} below. Before stating the lemma, we introduce some notation. For $p\in[0,1]$, decompose the configuration $Z$ into the triple $(\omega_p, Z^{\le p}, Z^{>p})$ defined as
\begin{equation*}
\omega_p = \omega_p(Z_\Lambda);
\hskip 0.45 in
Z^{\le p} =\begin{cases}
 Z& \text{if } Z \le p \\
1 & \text{otherwise} 
\end{cases};
\hskip 0.45 in
%Z^{>p} = \left\lbrace \begin{array}{ll}
% Z_\Lambda & \text{if } Z_\lambda > p \\
% 0 & \text{else} 
%\end{array}
%\right..
Z^{>p} = \begin{cases}
 Z & \text{if } Z> p \\
 0 & \text{otherwise} 
\end{cases}
.
\end{equation*}
Note that 
\begin{equation}
\label{e.omegaZZ}
\omega_p=\omega_p(Z^{\le p})=\omega_p(Z^{>p})\,,
\end{equation}
and that $Z$ can be recovered from the triple $(\omega_p, Z^{\le p}, Z^{>p})$.
\begin{lemma}\label{l.MonotoneMarkov}
%There exists a coupling $(Z_\Lambda^1, Z_\Lambda^2)$ such that for all $e\in E(\Lambda)$, 
%\[
%p\wedge Z_\Lambda^1(e) = p\wedge Z_\Lambda^2(e)\,,
%\]
%and furthermore, 
 Conditioned on the value of the first component $\omega_p$, the two other components $Z^{\le p}$  and  $Z^{>p}$ are conditionally independent.
%\begin{align*}
%(i) &   \left\lbrace \begin{array}{lll}
%\omega_p &= \omega_p(Z_\Lambda) & \\
%Z^{\le p} & =  Z_\Lambda & ,\text{if } Z_\Lambda \le p \\
%Z^{\le p} & = 1 & ,\text{else}  \\
%Z^{<p} & =  Z_\Lambda &,\text{if } Z_\lambda > p \\
%Z^{>p} & =  0 &,\text{else} 
%\end{array}
%\right.\\
%(ii) & \text{ conditioned on the value of the first component } \omega_p\,, \\
%& \text{ the two other components } Z^{\le p}\text{ and } Z^{>p}
%\text{ are conditionally independent}\,.
%\end{align*}
\end{lemma}

\proofof{the lemma} Fix $p\in [0,1]$ and omit it from the notation $\omega=\omega_p$ to make space for a time variable $t$. 

We basically follow the construction of the measure $\mu$ as the limiting measure of the Markov process $Z_t$, except that we divide the randomness used along the Markov chain into three components, the second and third being independent conditionally on the first one. Namely, define a Markov process 
$$(\omega_t, Z^{\le p}_t, Z^{>p}_t)_{t\geq 0} \in \{0,1\}^{E(\Lambda)} \times [0,1]^{E(\Lambda)} \times [0,1]^{E(\Lambda)},$$ 
where edges are updated at rate one, in such a way that the relations (\ref{e.omegaZZ}) between the three coordinates hold for all $t\geq 0$. To be consistent at $t=0$, the process starts either from the empty state $(\omega_0,Z^{\le p}_0,Z^{> p}_0)\equiv ({\bf 0},{\bf 1},\bf{1})$ or the full state $({\bf 1},{\bf 0},{\bf 0})$, where ${\bf 0}$ and ${\bf 1}$ denote the vectors all 0 and all 1 respectively. Then, instead of sampling $\mathcal{U}_e$ directly, let us proceed stepwise: first look whether $\omega_{t-}$ satisfies $\mathcal{D}_e$ or not. If it does, then 
let $\omega_t(e):=1$ with probability $p$. If $\omega_{t-} \notin \mathcal{D}_e$, then let $\omega_t(e):=1$ with probability $p/(p+(1-p)q)$.
This is exactly the heat-bath dynamics for $\phi^0_{G,p,q}$. Note that this part of the dynamics does not use at the two components $(Z^{\le p}, Z^{>p})$.

Let us describe how to update the component $Z^{\le p}_t$. If, after the update, $\omega_t(e)$ equals 0, then we fix $Z^{\le p}_t(e):= 1$. Otherwise (if $\omega_t(e)=1$), we use the following variable:
\[
T^{\le p}_e(Z^{\le p}) := \inf \{ u \in [0,p] : \omega^u(Z^{\le p}) \in \mathcal{D}_e \}\,.
\]
Note that $T^{\le p}_e(Z^{\le p})=T_e(Z)$ on the event $T_e(Z)\leq p$. Otherwise ({\em i.e.} $\omega_p(Z) \notin \mathcal{D}_e$), we set $T_e^{\le p} = p$. In either case, it is important here that no information about the third component $Z^{>p}$ has been used. 

Next, recall the update random variable $\mathcal{U}_e$ from the previous subsection (see \eqref{e.Ue}). It needed as an input the value of $T_e(Z_{t-})$. Let $\mathcal{U}_e^{\le p}$ be the same random variable here, with input the value of $T^{\le p}_e(Z_{t-}^{\leq p})$. Remembering that we are in the case $\omega_t(e)=1$,
update the value of $Z^{\le p}_t$ as follows, independently of everything:
\[
Z_t^{\le p}(e) \sim \mathcal{L} \bigl[ \mathcal{U}^{\le p}_e \,\big|\, \mathcal{U}^{\le p}_e \le p \bigr]\,,
\] 
where $\mathcal{L}$ stands for the law of the variable. We define $(Z^{>p}_t)$ in the same fashion, using $\mathcal{U}^{>p}_e$. In particular, the evolutions of $(Z^{\le p}_t)$ and $(Z^{> p}_t)$ are sampled out of the evolution of $(\omega_t)$ plus some randomness in each case that are independent of each other, hence the conditional independence of $Z^{\le p}$ and $Z^{>p}$ is satisfied. 

To conclude the proof, one just has to notice that if one defines
\begin{align*}
Z_t &:= \left\lbrace \begin{array}{ll} 
Z^{\le p}_t & \text{if } \omega_t(e)=1 \\
Z^{>p}_t & \text{else}\,,
\end{array}\right.
\end{align*}
then $(Z_t)_{t\geq 0}$ is exactly the Markov chain which was considered by Grimmett in \cite{\GrimmettCoupling}. 
(This is not hard to check; an important feature here is that if $T_e>p$, then the conditional law $\mathcal{L} \Bigl[ \mathcal{U}_e \md \mathcal{U}_e \le p \Bigr]$
does not depend on the exact value of $T_e$, and a similar thing holds for $\mathcal{U}^{>p}_e$ when $T_e\le p$.). In particular, from \cite{\GrimmettCoupling}, it converges to the 
unique invariant measure $\mu_\Lambda$, which inherits its conditional independence property. This finishes the proof of Lemma~\ref{l.MonotoneMarkov} and hence of Proposition~\ref{pr.MonotoneMarkov}. \QED

\begin{proposition}
This Markovian property extends to the infinite volume limit $\mu$ on $X = [0,1]^{E(\Z^2)}$. 
\end{proposition}

The same procedure works, provided that one is careful  with the initial state of the Markov chain.

\begin{remark}
The underlying dynamics is {\bf non-Fellerian}, and the limiting Markov process in the above theorem is derived from the monotonicity properties inherent to the dynamics. In particular, the relationship between this Markov process and its formal generator would need to be investigated. This seems to be a non-trivial task for the present dynamics. For this reason, we will not assume any explicit {\bf transition rule} for the infinite-volume dynamics $Z^\free_t$ (or $Z^\wired_t$) and will restrict ourselves to the ``compact case''.
\end{remark}

\subsection{Specific heat and self-organization}\label{s.SH}

\paragraph{Derivative of the edge-intensity of the random-cluster model.}

The first non-trivial effect which occurs in the near-critical random-cluster model is the fact that the derivative of the edge-intensity blows up around $p_c$ for $q\geq 2$. This implies that edges appear much faster in any monotone coupling near $p_c$ than they do for percolation.  
Define the edge-intensity function as follows: for all $p\in[0,1]$, let 
$\EI^{\fk_q}(p):= \phi_{p,q} \big( e\text{ is open} \big)\,,
$
where $e$ is any edge of $\Z^2$. It is not hard to check that at the critical (and self-dual) point $p_c(q)$, one has 
$
\EI^{\fk_q}(p_c) = \frac 1 2\,.
$ 

The quantity relevant to us here is the derivative in $p$ of the edge-density
$\frac{d}{dp}\EI^{\fk_q}(p)$. It corresponds to the average rate at which new edges appear in any possible monotone coupling $(\omega_{p,q})_{p\in[0,1]}$.
For integer $q\geq 2$, this quantity turns out to be linked to the so-called specific heat of the $q$-Potts model (The relationship between the derivative of the edge-intensity and the specific heat per site is detailed in \cite{\SpecificHeat}: for $q\geq 2$, they are within bounded factors from each other). 

This quantity is expected to behave like $|p-p_c(q)|^{\alpha(q)}$,
where $\alpha(q)= \frac{2(1-2u)}{3(1-u)}$ with $u=\frac{2}{\pi}\arccos(\sqrt q/2)$ and $q> 2$, and 0 for $q<2$. From this result (which is at the level of a prediction in the physics literature), it is reasonable to expect that if $\EI_n^{\fk_q}(p)$ denotes the edge-intensity for random-cluster model on $\Lambda_n=[-n,n]^2$ with wired boundary conditions, 
\[
\frac d {dp} \EI_n^{\fk_q}(p) \asymp n^{\alpha(q)}\,,
\]
as far as $n \lesssim L(p)$. In conclusion, as one raises $p$ near $p_c$, more edges will suddenly arrive. 

\begin{remark} For $q=2$, the results on the specific heat of the Ising model known since \cite{\Onsager, \FerdinandFisher} give that
\[
\frac d {dp} \EI(p) \sim a \log \frac 1 {|p-p_c|}\,, 
\]
as $p\to p_c$.  The finite-volume study of the specific heat (\cite{\Onsager, \FerdinandFisher}) leads to the following estimate:
let $\T_n$ be the torus $\Z^2 / n\Z^2$ and let  then
\[
\frac d {dp} \Bigm|_{p=p_c} \EI_n(p) \asymp  \log n \,.
\]
The extension to planar domains $\Omega_n:=  \frac 1 n \Z^2 \cap \, \Omega $ will be carried out in \cite{\SpecificHeat}, based on the recent results  from \cite{\ClementThesis} and \cite{\BdTperiodic, \BdTplane}.
\end{remark}

\paragraph{Near-critical behavior for $q\in(1,4]$ and self-organized monotone coupling.}

Let us consider the random-cluster model on $\Z^2$ with fixed cluster-weight $q\in(1,4]$ (we will drop it in some of the notation). In this paragraph, our goal is to illustrate that there must be a strong self-organized mechanism within the monotone Markovian coupling that goes beyond the specific heat effect.
To show this, we will take for granted the specific heat exponent and the pivotal exponent $\xi_4(q)$ given by 
$\alpha_4^{\fk_q}(n)=n^{-\xi_4(q)+o(1)}$, and based on this, we will estimate what would the correlation length exponent be if there was no self-organized mechanism (i.e., if new edges arrived in a Poissonian way, with an intensity measure that can depend on the current configuration but only up to bounded factors). 

Let us first describe our setup: we will restrict ourselves to a finite but very large box $\Lambda_n$ with wired boundary conditions. %Now, starting from a critical random-cluster configuration $\omega_{p_c}$ in $\mathbb T_n$, we raise $p=p_c +\Delta p$ until macroscopic effects start being non-negligible. If $p_0$ is the value where we stop, we should thus obtain the relation $L(p_0)\asymp n$.

We start from a critical configuration $\omega_{p_c}$ in $\Lambda_n$ and raise $p$ to the level 
$p=p_c + \Delta p$ in such a way that one still has $n \lesssim L(p)$. From the above discussion, one expects that about $n^2 \, \Delta p \, n^{\alpha(q) \wedge 1}$
new edges will arrive. If we assume the absence of self-organization, {\em i.e.}, if edges arrive more or less independently of the current configuration (except possibly a local rate which would depend on whether the endpoints of the edge are connected or not), then each of these arrivals should be macroscopic pivotal flips with probability about $n^{- \xi_4(q)}$ (we implicitly harnessed the fact that the pivotal exponent does not vary below the critical length). Therefore, at  $n\approx L(p)$, we expect
\[
n^2 \, \Delta p \, n^{\alpha(q) \wedge 1} n^{-\xi_4(q)} \approx 1\,.
\]
Let $L^\mathrm{Poiss}(p)$ denote the correlation length obtained via the above analysis. We find
\[
L^{\mathrm{Poiss}}(p) \approx  \Bigl( \frac 1 {|p-p_c|} \Bigr)^{\frac 1 {2 - \xi_4(q) + \alpha(q)\wedge 1}}\,.
\]
Now, $\xi_4(q)=\frac{5}{2}-\frac34u-\frac1{2-u}$ which allows to compute the predicted exponent $\nu^{\mathrm{Poiss}}(q)$ of $L^{\mathrm{Poiss}}(p)$ in terms of $u$ only. Similarly, the critical exponent $\nu(q)$ associated to the behavior of $L^{\fk_q}(p)$ is equal to $\frac{2-u}{3(1-u)}$. A simple computation shows that $\nu(q)>\nu^{\mathrm{Poiss}}(q)$ for $q>1$ which means that the Poissonian correlation length is much larger than the regular correlation length. 

This computation can be made rigorous for $q=2$, as explained in the introduction.

\paragraph{The hyperscaling relation between correlation length and specific heat.}

Let us finally remark that there is the following well-known {hyperscaling relation} between correlation length and specific heat (see, e.g., \cite{Henkel}): 
\begin{equation}
2-\alpha(q)=\nu(q)d\,, \label{e.hyperscl}
\end{equation}
which is expected to hold for all $q$, provided that the dimension $d$ of the underlying lattice $\Z^d$ is low enough.

We have been arguing that the mere quantity of new edges arriving does not explain the correlation length alone, but the self-organized structure in which they arrive also matters. Nevertheless, the hyperscaling relation tells us that the correlation length is in fact determined by the change in the edge intensity, in some other way. Unfortunately, we have not managed to relate to each other the mechanisms for the hyperscaling and for self-organization.

\subsection{Existence of emerging clouds}\label{ss.clouds1}

We now give a concrete manifestation in Grimmett's coupling of the self-organized behavior appearing in any monotone coupling of random-cluster models. We restrict ourselves to the finite case, since the transition rule for the infinite volume Markov process $(Z_t)_{t\geq 0}$ has not been established.
Let then $\Lambda=(V,E)$ be a finite box in $\Z^2$. 
Given a sample $Z = Z_\Lambda \in [0,1]^{E}$ from Grimmett's monotone coupling $\mu_\Lambda$, for an edge $e\in E$, let $\cloud(e)$ be the set of edges which appear simultaneously with $e$:
\[
\cloud(e):= \{ f\in E \text{ s.t. } Z(f)=Z(e) \}\,.
\]

The following proposition gives the first hint of some ``non-linear'' behavior:

\begin{proposition}\label{pr.emerging}
Fix $q>1$. For any $N\geq 1$, let $(\omega_p(Z_\Lambda)_{p\in[0,1]}$ be a monotone coupling in the box $\Lambda$. The probability that clouds of at least $N$ edges appear simultaneously in $\omega_p(Z_\Lambda)$ at some $p\in (0,1)$ converges to 1 when the size of the box $\Lambda\nearrow \Z^2$. \end{proposition}

This proposition is very easy to prove, yet one already sees here that the monotone Markovian coupling $(\omega_{\Lambda,p})_{0\le p \le 1}$ has a nature that is very different from the $q=1$ case.
\vskip 0.3 cm

\proof 
Let us consider the sets $E_1$, $E_2$ and $E_3$ in $\Lambda$ (which is assumed to be large enough) as defined in Fig.~\ref{fig:E}.

%\begin{wrapfigure}[3]{r}{2 in}
%\begin{center}
%\vspace{-0.5 in}
%\includegraphics[width=0.18 \textwidth]{E1E2E3}
%%\epsfysize=2 in \epsffile{E1E2E3.eps}
%\end{center}
%\end{wrapfigure}
%
%\begin{equation*}
%\begin{aligned}
%E_1& := \bigcup_{l=0}^n \big\{ \langle (0,l), (1,l) \rangle \big\}  \  \text{ (horizontal edges) }\\
%E_2& := \bigcup_{l=0}^{n-1} \big\{ \langle (0,l), (0,l+1) \rangle,\, \langle (1,l), (1,l+1) \rangle \big\} \ \text{ (vertical edges)} \\
%E_3& := \big\{\text{all the edges at distance $2$ from $E_1\cup E_2$} \big\} \setminus \bigl( E_1\cup E_2 \bigr) 
%\end{aligned}
%\end{equation*}

\begin{figure}[htdp]
\begin{center}
\vglue -0.2 cm
\begin{tabular}{cc}
\hskip 0 in \parbox{4in}{
\begin{align*}
E_1& := \bigcup_{l=0}^n \big\{ \langle (0,l), (1,l) \rangle \big\}\\
& \hskip 0.5 in \text{(horizontal inner edges)}\\
\ & \ \\
E_2& := \bigcup_{l=0}^{n-1} \big\{ \langle (0,l), (0,l+1) \rangle,\, \langle (1,l), (1,l+1) \rangle \big\}\\
&\hskip 0.5 in \text{(vertical inner edges)} \\
\ & \ \\
E_3& := \big\{\text{all edges neighboring $E_1\cup E_2$} \big\} \setminus \bigl( E_1\cup E_2 \bigr) 
\end{align*}}
&
\raise-1.1in\hbox{\includegraphics[width=0.2 \textwidth]{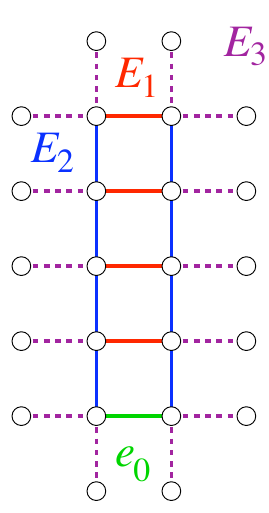}}
\end{tabular}
\end{center}
\vglue -0.5 cm
\caption{\label{fig:E}The definition of the sets $E_1$, $E_2$, $E_3$ and the edge $e_0\in E_1$. }
%\label{fig.E1E2E3}
\end{figure}%
%Let us choose $E\subset E(\Lambda)$ a finite family of edges $E:= \{ e_i = \langle x_i , y_i \rangle \}$ in such a way that there exists a finite subgraph
%$M=M_E \subset E(\Lambda) \setminus E$ which connects all the end-points $\{x_i\}\cup \{y_i\}$. 

Let us sample $Z_0 = Z^\Lambda_{t=0}$ according to the invariant measure $\mu_\Lambda$, and let us run the dynamics for a unit time. With positive probability, all edges in $E$ are updated and their labels at time 1 satisfy the following: all labels in $E_2$ are smaller than 1/4, the edge $e_0= \langle (0,0), (1,0) \rangle$ gets a label in $(1/4,1/2)$, and all other labels in $E_1 \cup E_3$ are larger than $3/4$. Under such circumstances, all edges $e\in E_1\setminus \{e_0\}$ are such that $T_e(Z_{t=1}) = Z_1(e_0)$. It could be that this situation evolves later on, but we have that, with positive probability, none of the edges in $E_2\cup E_3\cup \{e_0\}$ are updated from time 1 to time 2.
Knowing this, again with positive probability, all edges in $E_1$ are updated from time 1 to time 2 and all of them take exactly the value $u:= Z_{1}(e_0)$ (this is due to the Dirac mass $\delta_u$ in the law $\mathcal{U}_e$). Since we started at equilibrium, $Z_{t=2}$ has the equilibrium law, and edges in $E_2$ are all open or all closed in the projections of $Z_{t=2}$. This shows that with positive probability, at least $N$ edges appear simultaneously as one raises $p$.

  If $\Lambda$ is getting very large, we can divide the box into a lattice of $2N \times 2N$ squares. Starting from $Z_{\Lambda,0} \sim \mu_\Lambda$, the above strategy works in each box independently of what happens in other boxes. Stated like that, it looks wrong, since obviously 
   the dynamics itself is not independent from one square to another, but all that is needed in the above procedure is a positive lower bound on the probability 
   that this ``scenario'' happens. Using the structure of the dynamics, it is not hard to see that if $y_1, \ldots, y_K$ denote the indicator functions of the events that the scenario 
   happened in the squares $i\in\{1,\ldots, K\}$, then there is an independent product of Bernoulli $\eps>0$ variables which is stochastically dominated by our vector $(y_1, \ldots, y_K)$.
   In particular, the emergence of clouds is somewhat ergodic in the plane. 
   
     By changing slightly the argument, one can show that there are such clouds for any open interval of the variable $p\in [0,1]$. \QED
   
\paragraph{An intuitive explanation for the clouds.} We end this subsection by a hand-waving argument why these clouds of simultaneously opening edges appear and may play an important role in the dynamics of any monotone coupling. Consider a monotone coupling $(\omega_p, \omega_{p+\Delta p})$. Due to the factor $q^{\#\;  \text{clusters}}$ in the partition function,
FK configurations $\omega_p$ tend to have as many clusters as possible. Without this factor, one would be in the case of $q=1$, {\em i.e.}, standard percolation, and the edge intensity would be exactly $p$. With $q>1$, the random-cluster configuration tries to maximize the number of clusters, hence the edge-intensity drops to
a smaller value $\EI(p)<p$. In some sense, there is a fight between entropy (under the product measure $p^{\#\; \text{open edges}} (1-p)^{\#\; \text{closed edges}}$, most configurations have edge-intensity $p$)
and energy (which would correspond here to  $- \log (q^{\#\; \text{clusters}})$). When one goes from $p$ to $p+\Delta p$, new edges are added due to the entropy effect, but in such a way that not so many clusters will 
merge into a single one. A good strategy for adding many edges without a significant increase in energy is the following
%for doing this without loosing too much on the entropy side is to use the following
{\bf storing mechanism}.
Say we have two ``neighboring'' large clusters in $\omega_p$ with closed edges going from one to the other (these closed edges are then large-scale pivotal edges). Once we decide to open one of them,
it does not cost more energy to open a few others.

Now, we have just seen that the monotone coupling is Markovian in $p$: in particular, the only way for this storing mechanism to actually happen is  to have some values of $p$ where the system can simultaneously open several edges. This indeed can happen, due to the atom in the update distribution, as shown in Lemma~\ref{pr.emerging}, and the construction there was indeed a simple example of edges arriving simultaneously between two neighboring large clusters (the two components of $E_2$). 

It is worth noticing that this heuristic explanation (based on entropy/energy considerations plus the Markov property) hints that this ``non-linear phenomenon''  should be much stronger near the critical point. 
Indeed, near $p_c(q)$, there are many neighboring large clusters (i.e., many large scale pivotal points), which makes the {storing mechanism} more efficient. Away from criticality, this is not the case anymore. This intuition explains, for example, why we observe a blow-up of the derivative of the edge-intensity near $p_c$, and why the emerging clouds are more important there.

\subsection{Open questions on the structure of emerging clouds}\label{ss.clouds2}

\paragraph{Finite volume case.}
The previous subsection shows that there are non-trivial clouds with positive probability and the proof is not quantitative at all. It would be interesting to obtain information on the geometry of these clouds, which are witnesses of the long range dependency of the model. Let us consider the case of $G=\Lambda_n:= [-n,n]^2$ with {wired} boundary conditions. We strongly suspect the following behavior:

\begin{question}[Macroscopic clouds near $p_c$]\label{q.macroCloud}
For all $n\geq 1$, with $\mu_{\Lambda_n}$-probability at least a universal constant $c>0$, there is at  least one {\bf macroscopic cloud} in $\Lambda_n$, i.e., whose diameter is larger than $cn$. Furthermore, with probability going to 1 as $n\to \infty$, the labels  of such macroscopic clouds concentrate around the critical value $p_c(q=2)$.
\end{question}

To answer such a question, it is natural to run the dynamics at equilibrium (i.e., $Z_0^n \sim \mu_{\Lambda_n}$) 
for a short amount of time that is given precisely by the rescaling
$\tau_n:= [n^{2} \alpha_4^{\fk_q}(n)]^{-1}$.
Doing so, only finitely many macroscopic {pivotal} edges will be resampled, and it is easy to convince ourselves that with positive probability at least two of them will pick the same label thus creating a macroscopic {cloud}. For $q=2$, this intuition is close to being rigorous, since we have at our disposal a
'stability property' from the forthcoming \cite{\DFKSL} which suggests that the ``geometry'' of $Z^n_{t=\tau_n}$ could be recovered with high precision from $Z^n_0$ plus
the updates of the initially macroscopically important edges (neglecting the ``smaller'' updates).  However, the stability  result holds only for $\omega_{p_c}$, not the entire coupling $Z$. Thus, a certain control on the {\it concentration} of the labels around $p_c$ would be helpful for both parts of Question~\ref{q.macroCloud}.

The intuition that big clouds should appear only around the critical point can be translated into the following conjecture:

\begin{question}[Local clouds away from $p_c$]
For any $\delta>0$, emerging clouds with labels outside of $(p_c-\delta, p_c+\delta)$ are {\bf local}
in the sense that the largest such cloud in $\Lambda_n$ should be of logarithmic size.
\end{question}

A natural way to attack this question would be via a {\bf coupling} argument. Namely, construct a coupling $(Z_{\Lambda_n}^{\geq p_c + \delta}, \tilde Z_{\Lambda_n}^{\geq p_c + \delta})$ (see the notation in Subsection~\ref{ss.coupling}) whose marginals are $\mu_{\Lambda_n}^{\geq p_c + \delta}$, and whose coordinates 
are identical on a small neighborhood of the origin, but with probability at least $\lambda^k$ (with $\lambda\in (0,1)$), are independent of each other outside a box of size $k$ (an exponential decay of correlations). Such a statement is proved for the supercritical (or subcritical) random-cluster measure $\phi_{\Z^2,p,q}$, yet the lack of a DLR (spatial Markov) property for our monotone coupling $\mu_{\Lambda_n}$ prevents us for extending the result to the coupling in an obvious way.

Finally, it would be interesting to prove quantitative results on the size of emerging clouds in the finite volume case ($\Lambda_n$). This question is further discussed in \cite{\SpecificHeat}.

\paragraph{Infinite volume case.} Clouds are well-defined objects: since there is a unique limiting measure $\mu_{\Z^2}$ of Grimmett's monotone coupling and for any  $e\in E(\Z^2)$, $\cloud(e)$ can still be defined relatively to a sample $Z$ from  $\mu_{\Z^2}$. 

\begin{question}
Prove that a.s.~there exist non-trivial emerging clouds.
\end{question}

This does not follow directly from the existence of non-trivial clouds in the finite volume case.
Assuming the above question, the next natural question would be the following:

\begin{question}
Is it the case that emerging clouds are a.s.~finite when $q\le4$?
\end{question}

Note that for large $q$, an infinite number of edges appear at $p_c(q)$. 

\section{What about the influence of an edge?}\label{ss.influence}

Let us mention an alternative approach to a ``geometric'' understanding of the near-critical random-cluster model.
As a continuation of the work by Kesten on near-critical percolation \cite{\KestenScaling}, Russo's formula should be replaced by a slightly different formula. Fix $q\geq 1$, $\ep>0$, and an increasing event $A$. Then (see \cite{\GrimmettFK}),
 \begin{equation*}
    \frac{d}{dp}\phi_{G,p,q}^\xi(A)\asymp\sum_{e\in E}I_A^p(e),
  \end{equation*}
  where the constants in $\asymp$ depend on $q$ and $\ep$ only. Above, $I_A^p(e)$
denotes the \emph{(conditional) influence} on $A$ of the edge $e\in E$ defined by \begin{equation}\label{russo_influence} I_A^p(e) := \phi_{G,p,q}^\xi (A|e\text{ is open}) - 
\phi_{G,p,q}^\xi (A|e\text{ is closed}). \end{equation}
It is tempting to use this extension of Russo's formula to see what our results on the correlation length (Theorem \ref{thm:correlation length})
may imply on the {\it influences} $I_A^p(e)$.
To avoid boundary issues, let us consider the case of the torus $\T_n:= \Z^2/ n \Z^2$, and let $A_n$ be the event that 
there is an open circuit with non-trivial homotopy in $\T_n$. It is easy to check (by self-duality) that $\phi_{p_c,2} \big( A_n \big) \le 1/2$.
The results from Section \ref{sec:proofs} can easily be generalized to the torus.
In particular, there exists a constant $\lambda>0$ such that if 
%CHANGED SQRT
$p_n:=p_c(2)+ \lambda \frac{\log n}{n}$, then
\[
\phi_{p_n, 2} \big( A_n\big) \geq 3/4\,.
\]

Using \eqref{russo_influence}, this says that 
\[
\int_{p_c}^{p_c+ \lambda \frac {\log n}{ n}}  \big( I_{A_n}^p(e_{hor}) + I_{A_n}^p(e_{ver}) \big) dp \geq \Omega(1) \frac 1 {n^2}\,,
\]
where $e_{hor}$ and $e_{ver}$ are any horizontal and vertical edges in $\T_n$. Since it is natural to expect that 
on the interval 
%CHANGED SQRT
$[p_c, p_c+\lambda \frac {\log n}{n}]$, {\it influences} behave smoothly (proving such a statement could require to understand the geometry of the clouds mentioned in the previous section), the following conjecture should hold.
\begin{conjecture}
For any 
%CHANGED SQRT
$n\geq 1, \lambda>0, \, p\in [p_c -\lambda\frac {\log n} n, p_c+\lambda\frac {\log n} n]$ and any $e\in \T_n$,
%CHANGED SQRT 
\[
I_{A_n}^p(e) \geq c\, \frac 1 {n \log n}\,,
\]
where $c=c(\lambda)$ is some positive constant.
\end{conjecture}

In fact since it is reasonable to conjecture that in Theorem \ref{thm:correlation length}, one has actually
$L_{\rho,\ep}^\xi(p)\asymp |p-p_c|^{-1}$, one may strengthen the previous conjecture into 
the following one:

\begin{conjecture}
For any $n\geq 1, \lambda>0, \, p\in [p_c -\frac \lambda n, p_c+\frac \lambda n]$ and any $e\in \T_n$, 
\[
c\, \frac 1 {n} < I_{A_n}^p(e)  < c^{-1} \, \frac 1 {n} \,,
\]
where $c=c(\lambda)$ is some positive constant.
\end{conjecture}

\paragraph{Acknowledgments.} We wish to thank Vincent Beffara, Geoffrey Grimmett, Alan Hammond, Leonardo Rolla, Alan Sokal and Wendelin Werner for very inspiring conversations. We also wish to thank Alan Hammond for carefully reading the article.

Some of the starting ideas of this work were found during the Discrete Probability Semester in 2009 at the Mittag-Leffler Institute in Sweden. 
The work of HDC was supported by the ANR grant
BLAN06-3-134462, the EU Marie-Curie RTN CODY, the ERC AG CONFRA, and the
Swiss FNS.
The work of CG was supported by the ANR grant BLAN06-3-134462.
The work of GP was supported by an NSERC Discovery Grant at the University of Toronto, an EU Marie Curie International Incoming Fellowship at the Technical University of Budapest, by the Hungarian National Science Fund OTKA grant K109684, and by the MTA R\'enyi ``Lend\"ulet'' Limits of Discrete Structures Research Group.

\small \bibliographystyle{alpha}
\bibliography{NearCritical.ref}
\normalsize
%\addcontentsline{toc}{section}{References}% AJOUTER BIBLI A LA TABLE DES MATIERES !
%\bibliography{mr,prep,notmr}
%%UMPA
%\bibliography{/Users/christophe/Desktop/Maths/Bibtex/mr,/Users/christophe/Desktop/Maths/Bibtex/prep,/Users/christophe/Desktop/Maths/Bibtex/references}

%%% Chez moi
%\bibliography{/Users/Christophe/Desktop/Maths/Bibtex/mr,/Users/Christophe/Desktop/Maths/Bibtex/prep,/Users/Christophe/Desktop/Maths/Bibtex/references}
%%Ne pas oublier de Fichiers Bib ici 
%% Si conflit ! Alors il prends en compte le premier des fichiers Bib seulement.

\ \\
{\bf Hugo Duminil-Copin}\\
Universit\'e de Gen\`eve\\
\url{http://www.unige.ch/~duminil/}\\
\\
{\bf Christophe Garban}\\
CNRS \& Ecole Normale Sup\'erieure de Lyon, UMPA \\
46 all\'ee d'Italie\\
69364 Lyon Cedex 07 France \\
\url{http://perso.ens-lyon.fr/christophe.garban}\\
\\
{\bf G\'abor Pete}\\
R\'enyi Institute of the Hungarian Academy of Sciences, and\\
Institute of Mathematics, Technical University of Budapest,\\
Egry J\'ozsef u. 1, Budapest 1111, Hungary\\
\url{http://www.math.bme.hu/~gabor}\\

\end{document}